\documentclass[a4paper,10pt,reqno]{amsart}

\usepackage{graphicx}
\usepackage{ulem}
\usepackage{here}
\usepackage{mathtools}
\usepackage[margin=25truemm]{geometry}

\makeatletter
\def\@settitle{\begin{center}
  \baselineskip14\p@\relax
  \normalfont\LARGE\bfseries

  \@title
  \ifx\@subtitle\@empty\else
     \\[1ex] 
     
     \normalsize\mdseries\@subtitle
  \fi
 \ifx\@didication\@empty\else
     \\[2ex] 
     
     \large\mdseries\it\@dedication
  \fi
  \end{center}
}
\def\subtitle#1{\gdef\@subtitle{#1}}
\def\@subtitle{}
\def\dedication#1{\gdef\@dedication{#1}}
\def\@dedication{}

\makeatother

\newcommand{\incl}[1][r]
  {\ar@<-0.2pc>@{^(-}[#1] \ar@<+0.2pc>@{-}[#1]}
\newcommand{\eq}[1][r]
   {\ar@<-3pt>@{-}[#1]
    \ar@<-1pt>@{}[#1]|<{}="gauche"
    \ar@<+0pt>@{}[#1]|-{}="milieu"
    \ar@<+1pt>@{}[#1]|>{}="droite"
    \ar@/^2pt/@{-}"gauche";"milieu"
    \ar@/_2pt/@{-}"milieu";"droite"}

\newcommand\FB{F_{\mathcal{B}}}
\newcommand\GT{\widehat{GT}}
\newcommand\B{\widehat{B}}
\newcommand\Gone{\Gamma_{g,1}}
\newcommand\GGone{\widehat{\Gamma}_{g,1}}
\newcommand\GGzero{\widehat{\Gamma}_{g,0}}
\newcommand\GGn{\widehat{\Gamma}_{g,n}}

\newtheorem{thm}{Theorem}[section]
\newtheorem{cor}[thm]{Corollary}
\newtheorem{prop}[thm]{Proposition} 
 
\newtheorem{defn}[thm]{Definition}

\newtheorem{lem}[thm]{Lemma}
\makeatletter
\renewcommand{\section}{\@startsection
{section}{1}{0mm}{5mm}{2mm}{\raggedright\bfseries}}

\@addtoreset{equation}{section}

\makeatother

\begin{document}

\title[The Grothendieck-Teichm\"uller group $\GT$ acts on 
$\widehat{\Gamma}_{g,0}$ and $\widehat{\Gamma}_{g,1}$]{The Grothendieck-Teichm\"uller group $\GT$ acts on the 
genus $g$ mapping class group with $0$ or $1$ marked point}

\markboth{P.Lochak, H.Nakamura, L.Schneps}
{}

\author{Pierre Lochak}
\address{
 Institut de Math\'ematiques de Jussieu \\
 Sorbonne Universit\'e\\
 Paris , France }
 \email{Pierre.Lochak@imj-prg.fr }

  \author{Hiroaki Nakamura}
 \address{ Department of Mathematics, 
    University of Osaka, 
    Toyonaka, Osaka 560-0043, Japan}
\email{nakamura@math.sci.osaka-u.ac.jp }
  
\author{Leila Schneps}
\address{
 Institut de Math\'ematiques de Jussieu \\
 Sorbonne Universit\'e\\
 Paris , France }
\email{Leila.Schneps@imj-prg.fr }

\keywords{Mapping Class Groups, Grothendieck-Teichm\"uller Group }
\subjclass{Primary 14E20, Secondary 20F34}

\maketitle

\begin{abstract}
The goal of this article is to prove that the
Grothendieck-Teichm\"uller group $\GT$ acts on 
$\GGzero$ and  $\GGone$,
the (full) profinite genus $g$ mapping class group with $0$ or $1$ marked point, for every $g>0$.
\end{abstract}


\setcounter{tocdepth}{1}
\tableofcontents

\section{Main Theorem}

Consider the simple loops on a topological surface $\Sigma_{g,m}^n$
of genus $g$ with $n$ marked points and $m$ boundary components shown in the following diagram:
\begin{figure}[h!]
    \centering
  \includegraphics[width=0.8\textwidth]{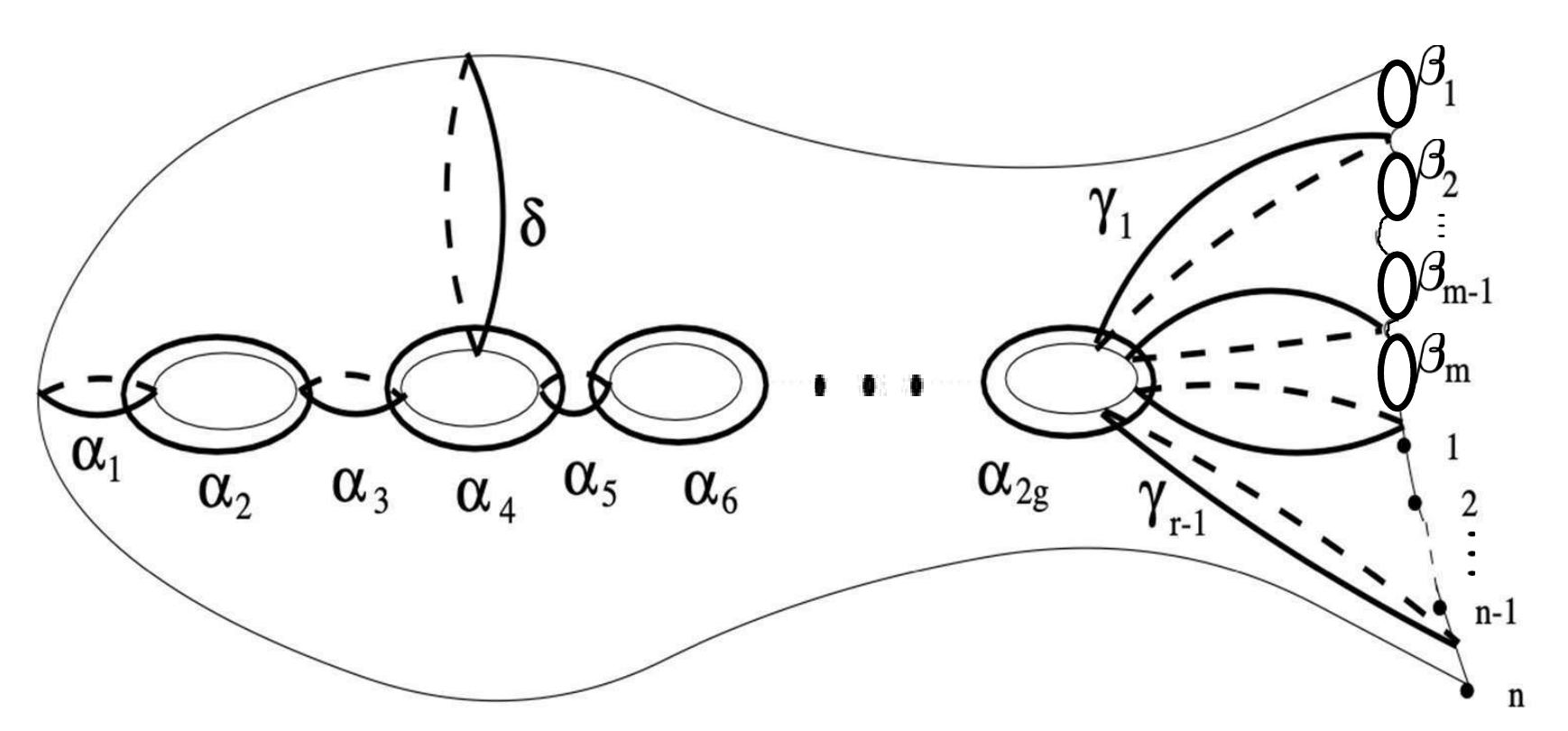}
  
    \caption{Simple loops on $\Sigma_{g,m}^n$ ($r=m+n$) whose Dehn twists generate 
the pure mapping class group}
    \label{fig:Fish}
\end{figure}

For simplicity, we call simple loops (simple closed curves) on a surface 
just {\it loops}. In principle, we will use Greek letters for loops and Roman letters for the Dehn twists
along them.
Let $m=0$. 
Wajnryb \cite{W} and Gervais \cite{G} gave a finite presentation
of the pure mapping class group $\Gamma_{g,n}$, generated by
the Dehn twists $a_i$ along $\alpha_i$ 
for $i=1,\ldots,2g$, $d$ along $\delta$, 
and $g_i$ along $\gamma_i$
for $i=1,\ldots,n-1$. In this article we mainly restrict
our attention to the case $n=1$, where the loops
$\gamma_1,\ldots,\gamma_{n-1}$ are absent; we give the complete 
finite presentation in this case in Theorem \ref{wajnryb} below.

\vspace{.3cm}
We denote the profinite completion of $\Gamma_{g,1}$ by $\GGone$; it 
is topologically presented by the same generators and relations as the
discrete version. It is well-known that the Grothendieck-Teichm\"uller
group $\GT$ acts on the genus zero pure mapping class group 
$\widehat\Gamma_{0,n}$ for $n\ge 4$, and the action obeys a ``lego'' rule
suggested by Grothendieck in \cite{Gr}. The lego rule is not special to
genus 0; it is spelled out explicitly in \cite{HLS} for surfaces of
all types $(g,n)$. However, it was shown in \cite{HLS} and \cite{NS} that
in order for elements of $\GT$ to act as automorphisms of 
$\widehat\Gamma_{g,n}$ respecting the lego rules, a new condition
(first introduced in \cite{LNS}) must be added to the
three defining conditions of $\GT$. It is unclear,
however, whether the additional condition actually defines a proper subgroup
of $\GT$, and it is also not known whether the full group $\GT$ is an 
automorphism of $\widehat\Gamma_{g,n}$ in some manner not necessarily 
subject to the lego rules. Apart from these works proving that a certain
(possibly full) subgroup of $\GT$ could be identified with the automorphism
group of $\widehat\Gamma_{g,n}$ respecting the lego, it was not known
whether and how $\GT$ could be identified with an automorphism group of
$\widehat\Gamma_{g,n}$, except in the particular cases of 
$\widehat\Gamma_{1,1}$ and $\widehat\Gamma_{1,2}$, whose close similarities
with the genus zero groups $\Gamma_{0,[4]}$ (the full genus zero mapping
class group) and $\Gamma_{0,5}$ made it possible to show that $\GT$
did act on these groups, albeit without necessarily respecting the lego (in the
absence of the additional relation) for $\widehat\Gamma_{1,2}$.

The main result of this paper states that in fact, $\GT$ gives an automorphism
group of a tower of the profinite mapping class groups $\GGzero$ and $\GGone$ for all 
$g\ge 1$.  According to the finite presentations given by
Wajnryb \cite{W}, the groups $\GGone$ and $\GGzero$ have the same generators
$d,a_1,\ldots,a_{2g}$, Dehn twists along the corresponding simple loops
$\delta,\alpha_1,\ldots,\alpha_{2g}$ shown in Figure~1,
and the presentation of $\GGzero$ is the same as that of $\GGone$ with
a single additional relation. We recall the presentation of $\GGone$ in
Theorem \ref{wajnryb}, and the additional defining relation of $\GGzero$ is
recalled in Theorem \ref{wajnryb-ii}.

\begin{thm}\label{thm11} Let $F=(\lambda,f)\in \GT$. 
The following action of $F$ given on the generators $d,a_1,\ldots,a_{2g}$
of either one of the profinite groups $\GGzero$ or $\GGone$
extends to an automorphism of that group:
\begin{equation}\label{Faction}
\begin{cases}
F(a_1)=a_1^\lambda\\
F(d)=d^\lambda\\
F(a_i)=f(a_i^2,y_i)a_i^\lambda f(y_i,a_i^2)&\hbox{for } 2\le i\le 2g,
\end{cases}
\end{equation}
where $y_i=a_{i-1}\cdots a_1^2\cdots a_{i-1}$ for $2\le i\le 2g$.
\end{thm}

\vspace{.3cm}
In the next section we recall some background material, essentially 
the presentation of the mapping class group
$\Gamma_{g,1}$ by generators $d,a_1,\ldots,a_{2g}$ with Wajnryb's
relations labeled (A), (B), (C) and (C$'$) in Theorem \ref{wajnryb}, together
with the $\GT$ automorphism action on these groups given by Drinfeld
(\cite{D}). In Sections 3 and 4 we prove the main theorem in the
case of $\widehat\Gamma_{2,1}$, which only requires the action of $F\in \GT$ 
proposed in Theorem \ref{thm11} to respect the relations 
(A) and (B). In Section 5 we settle the case of $\widehat\Gamma_{3,1}$ 
by showing that the purported $\GT$-action respects relation (C), and
prepare the ground for the general case $g\ge 4$, which is finally dealt 
with in Section 6 by verifying that relation (C$'$) is also respected. 
Finally, Section 7 contains the proof 
of the main theorem for the group $\GGzero$, by showing that the $\GT$-action 
respects the one additional relation (D) needed to pass from 
$1$ to $0$ marked points. Section 8 contains some perspectives for future
work.  The authors note that a large part of the calculations and proofs below 
stem from our unpublished discussions which led to the short paper \cite{LNS} 
in 1997.

\section{Some background material in genus zero}

We recall here some known facts about the $\widehat{GT}$-action on the full profinite
Artin braid group $\B_n$ on $n$ strands.  We first have the original result due
to Drinfeld (in [D]).

\begin{thm}\label{drinfeldaction}
Let $F=(\lambda,f)\in \widehat{GT}$. Then $F$ induces an automorphism of the
profinite Artin braid group $\B_n$ for all $n\ge 3$ via the formulas
\begin{equation}\label{drinfeld}
F(\sigma_1)=\sigma_1^\lambda,\ \ \ \ 
F(\sigma_i)=f(\sigma_i^2,\eta_i)\sigma_i^\lambda f(\eta_i,\sigma_i^2)\ \ \hbox{for}\ \ 2\le i\le n-1.
\end{equation}
This action also satisfies
$$F(\eta_i)=\eta_i^\lambda,\ \ \ F(\omega_i)=\omega_i^\lambda\ \ \hbox{for}\ i=2,\ldots,n$$
where $\eta_i=\sigma_{i-1}\cdots \sigma_1\cdot \sigma_1\cdots \sigma_{i-1}$
and $\omega_i=(\sigma_1\cdots \sigma_{i-1})^i$.
\end{thm}

\vskip .2cm
\begin{defn} {\rm
A {\it bracketing} on $n$ points consists of $k$ pairs of brackets with
$1\le k\le n-2$ in which no pair of brackets contains all the points, 
and if a pair of brackets contains one bracket from another pair, then it
contains both brackets from that pair (see examples below).
A {\it maximal bracketing} consists of $n-2$ pairs of brackets. An 
{\it A-move} on a bracketing consists in replacing a configuration of 
brackets $(X_1,X_2),X_3$ by the configuration 
$X_1,(X_2,X_3)$
or vice versa, where $X_i$ denotes either a single point or a
set of consecutive points included in a bracket (which itself may contain
other, smaller brackets). }
\end{defn}

\noindent {\bf Examples 2.2.1.}  Here are two examples of bracketings of $x_1,\ldots,x_6$:
$((x_1,x_2),(x_3,x_4)),(x_5,x_6)$ and $(x_1,x_2),(x_3,x_4),(x_5,x_6)$. 
A {\it maximal} bracketing of 6 points contains $n-2=4$ pairs of brackets, 
so we could take for example 
$$(((x_1,x_2),(x_3,x_4)),x_5),x_6 \quad{\rm or}\quad ((x_1,x_2),(x_3,x_4)),(x_5,x_6).$$
These two maximal bracketings differ by one A-move. The maximal bracketings:
$$(((x_1,x_2),(x_3,x_4)),x_5),x_6 \quad{\rm and}\quad  (((x_1,x_2),x_3),x_4),(x_5,x_6)$$
differ by two commuting A-moves. 

\begin{thm}\label{Amovetheorem} 
Let $n\ge 3$. Let $F\in \widehat{GT}$ and let 
$\overline{F}$ denote the outer automorphism of $\B_n$ induced by the 
automorphism $F$ acting as in \eqref{drinfeld}.  Then 

\begin{enumerate}
\item[(i)]

for each maximal 
bracketing ${\mathcal{B}}$ of $n$ points, there is a lift of
$\overline{F}$ to an automorphism $F_{\mathcal{B}}$ of 
$\B_n$ having the property that 
if $i,\cdots,j$ denotes a consecutive set of points enclosed within
a pair of brackets (regardless of smaller sets enclosed by brackets with
the pair), then denoting by $T_{i,\cdots,j}$ the full twist of the braid 
packet $i,\ldots,j$, $F_{\mathcal{B}}$ acts on
$T_{i,\ldots,j}$ by $F(T_{i,\ldots,j})=T_{i,\ldots,j}^\lambda$.

\item[(ii)]

Let $X_1=(x_i,\ldots,x_j)$ and $X_2=(x_{j+1},\ldots,x_k)$
be two consecutive packets of consecutive strands in $\B_n$. Let
$\sigma_{X_1,X_2}$ denote the flat crossing of the packet $X_1$ to the
right over the packet $X_2$ (so that the full twist $T_{i,\ldots,k}$ is
given by $T_{i,\ldots,j}=\sigma_{X_1,X_2}^2T_{i,\ldots,j}T_{{j+1},\ldots,k}$).
If ${\mathcal{B}}$ is a bracketing containing the configuration
$(X_1,X_2)$, then the associated automorphism of $\B_n$ acts by
\begin{equation}\label{flatcrossing}
F_{\mathcal{B}}(\sigma_{X_1,X_2})=\sigma_{X_1,X_2}^\lambda.
\end{equation}

\item[(iii)]

If two bracketings ${\mathcal{B}}$ and ${\mathcal{B}'}$ differ
by a single A-move which takes a pair of brackets containing
the consecutive set of points $i,\ldots,j$ to a pair containing
the consecutive set of points $i',\ldots,j'$, then the corresponding 
automorphisms $F_{\mathcal{B}}$ and $F_{\mathcal{B}'}$ are related by
\begin{equation}\label{Amove}
F_{\mathcal{B}'}={\rm inn}\,f(T_{i,\ldots,j}, T_{i',\ldots,j'})
\circ F_{\mathcal{B}}.
\end{equation}
\end{enumerate}
\end{thm}

\vspace{.3cm}
\noindent {\bf Example 2.3.1.} The {\it standard bracketing} ${\mathcal{B}_0}$
consists of the pairs of brackets containing 
the points $1,\ldots,i$ for $2\le i\le n-1$.

\begin{lem} The Drinfeld action of $F\in \GT$ on the profinite Artin braid group $\B_n$ given in 
\eqref{drinfeld} is the automorphism $F_{\mathcal{B}_0}$ associated to the standard
bracketing ${\mathcal{B}_0}$. 
\end{lem}

\begin{proof}
In view of Theorem \ref{Amovetheorem} (i), 
$F_{\mathcal{B}_0}$ acts via $\lambda$ on $\sigma_1$ and $\omega_i$ for 
$2\le i \le n$. To determine the action of $F_{\mathcal{B}_0}$ on the other 
generators $\sigma_i$ of $\B_n$, we note that for each $i>1$, the modification 
of ${\mathcal{B}_0}$ by erasing the single pair of brackets surrounding points 
$1,\ldots,i$ and replacing it by the pair of brackets surrounding the two points
$i$ and $i+1$ is an A-move. Letting ${\mathcal{B}_i}$ denote the 
bracketing obtained by this A-move for $i=2,\ldots,n-1$,
Theorem \ref{Amovetheorem} (iii) tells us that $F_{\mathcal{B}_0}
={\rm inn}\,f(T_{1,\ldots,i},\sigma_i^2)\circ {\mathcal{B}_0}$. The full 
twist $T_{i,i+1}$ on a pair of consecutive strands is $\sigma_i^2$,
and the full twist $T_{1,\ldots,i}$ is equal to $\omega_i$. In the braid group 
$\B_n$, the $\omega_i$ and $\eta_i$ all commute among themselves and with each other and
satisfy the identities $w_{i-1}\eta_i=\omega_i$. Furthermore $w_{i-1}$ commutes
with $\sigma_i$. So $f(\omega_i,\sigma_i^2)=f(w_{i-1}\eta_i,\sigma_i^2)=
f(\eta_i,\sigma_i^2)$ (since $w_{i-1}$ commutes with both $\eta_i$ and $\sigma_i^2$
and $f$ is a (profinite) commutator). Therefore the action of
$F_{\mathcal{B}_0}$ is exactly the action given in \eqref{drinfeld}.
\end{proof}

Finally, we have the following proposition, which is a special case 
of a much more general principle, which we will need in \S\S 5,6 below.

\vspace{.2cm}
\begin{thm}\label{doublestrandsubgroup} Let $n=8$, let $X_1=(x_1,x_2)$, $X_2=(x_3,x_4)$,
$X_3=(x_5,x_6)$ and $X_4=(x_7,x_8)$, and fix a maximal bracketing
$${\mathcal{B}}=((X_1,X_2),X_3),X_4=(((x_1,x_2),(x_3,x_4)),(x_5,x_6)),
(x_7,x_8).$$
Set

\vspace{-0.8cm}
\begin{equation}\label{t12etc}
\begin{cases}
\tau_1=\sigma_{X_1,X_2}=\sigma_2\sigma_1\sigma_3\sigma_2\\
\tau_2=\sigma_{X_2,X_3}=\sigma_4\sigma_3\sigma_5\sigma_4\\
\tau_3=\sigma_{X_3,X_4}=\sigma_6\sigma_5\sigma_7\sigma_6.
\end{cases}
\end{equation}
Then 

\vspace{.1cm}\noindent 
(i) The subgroup $\langle \tau_1,\tau_2,\tau_3\rangle\subset
\B_8$ is isomorphic to $\B_4$ and the action of
$F_{\mathcal{B}}$ on that subgroup is the Drinfeld action \eqref{drinfeld}
with $\sigma_i$ replaced by $\tau_i$.

\vspace{.1cm}\noindent 
(ii) $F_{\mathcal{B}}$ is related to the standard Drinfeld automorphism $F=F_{\mathcal{B}_0}$ 
by 
$$F_{\mathcal{B}}={\rm inn}\,\bigl(f(\eta_3,\sigma_3^2)f(\eta_5,\sigma_5^2)f(\eta_7,\sigma_7^2)\bigr)\circ F.$$
\end{thm}

\begin{proof} (i) The packets $X_1$ and $X_2$ appear in ${\mathcal{B}}$, 
and $\tau_1=\sigma_{X_1,X_2}$, so by \eqref{flatcrossing} we have
$F_{\mathcal{B}}(\tau_1)=\tau_1^\lambda$. For $i=2,3$, 
let $Y_i=\tau_i\cdots \tau_1^2\cdots \tau_i$ .

Consider the maximal bracketings
\begin{equation}\label{threebracketings}
\begin{cases}
{\mathcal{B}}=(((x_1,x_2),(x_3,x_4)),(x_5,x_6)),(x_7,x_8)\\
{\mathcal{B}'}=((x_1,x_2),((x_3,x_4),(x_5,x_6))),(x_7,x_8)\\
{\mathcal{B}''}=((x_1,x_2),(x_3,x_4)),((x_5,x_6),(x_7,x_8)).
\end{cases}
\end{equation}
The second one differs from the first by the single A-move changing the pair 
of brackets containing $1,2,3,4$ to the pair containing $3,4,5,6$, 
and the third one differs from the first by the single A-move changing
the pair of brackets containing $1,2,3,4,5,6$ to the pair containing
$5,6,7,8$.  Therefore by Theorem \ref{Amovetheorem}, we have
\begin{equation}\label{inns}
\begin{cases}
F_{\mathcal{B}'}={\rm inn}\,f(T_{1234},T_{3456})\circ F_{\mathcal{B}}\\
F_{\mathcal{B}''}={\rm inn}\,f(T_{123456},T_{5678})\circ F_{\mathcal{B}}.
\end{cases}
\end{equation}
By Theorem \ref{Amovetheorem} (ii), we have
\begin{equation}\label{tausigmas}
\begin{cases}
T_{1234}=\sigma_{X_1,X_2}^2\sigma_1^2\sigma_3^2=\tau_1^2\sigma_1^2\sigma_3^2\\
T_{3456}=\sigma_{X_2,X_3}^2\sigma_3^2\sigma_5^2=\tau_2^2\sigma_3^2\sigma_5^2\\
T_{5678}=\sigma_{X_3,X_4}^2\sigma_5^2\sigma_7^2=\tau_3^2\sigma_5^2\sigma_6^2.
\end{cases}
\end{equation}

Since $X_3$ and $X_4$ appear in ${\mathcal{B}'}$ and
$\sigma_{X_3,X_4}=\tau_2$, we have
$F_{\mathcal{B}'}(\tau_2)=\tau_2^\lambda$
by \eqref{flatcrossing}. 
Since $\sigma_1^2$, $\sigma_3^2$ and $\sigma_5^2$ commute with $T_{1234}$ and
$T_{3456}$ and $Y_2=\tau_1^2$, we find\footnote{This standard commutation
result is spelled out in (ii) of the Haiku Lemma in Section 5 below.}
$$f(T_{1234},T_{3456})=f(\tau_1^2\sigma_1^2\sigma_3^2,\tau_2^2
\sigma_3^2\sigma_5^2)=f(\tau_1^2,\tau_2^2)=f(Y_2,\tau_2^2).$$
Thus 
$$F_{\mathcal{B}}(\tau_2)=f(\tau_2^2,Y_2)\tau_2^\lambda f(Y_2,\tau_2^2).$$

It remains only to compute $F_{\mathcal{B}}(\tau_3)$ using \eqref{tausigmas}
and the second line of \eqref{inns}. For this, we make the full twist
$T_{123456}$ explicit: 
$$T_{123456}=\sigma_{X_1,X_2}^2\sigma_{X_2,X_3}\sigma_{X_1,X_2}^2
\sigma_{X_2,X_3}\sigma_1^2 \sigma_3^2\sigma_5^2$$
(best checked by braiding). Since $\sigma_1^2$, $\sigma_3^2$ and
$\sigma_5^2$ commute with $T_{123456}$ and $T_{5678}$ we have
$$f(T_{123456},T_{5678})=f\bigl(\sigma_{X_1,X_2}^2\sigma_{X_2,X_3}
\sigma_{X_1,X_2}^2\sigma_{X_2,X_3},\sigma_{X_3,X_4}^2\bigr).$$
Because $\sigma_{X_1,X_2}^2$ commutes with both arguments, we can remove the
left-hand multiplication by this term in the left argument to get
$f(\sigma_{X_2,X_3}\sigma_{X_1,X_2}^2\sigma_{X_2,X_3},\sigma_{X_3,X_4}^2)$. 
Using
$\tau_2=\sigma_{X_2,X_3}$ and $\tau_3=\sigma_{X_3,X_4}$, we can then write
\begin{equation}\label{goodform}
f(T_{123456},T_{5678})=f(\tau_2\tau_1^2\tau_2,\tau_3^2)=f(Y_3,\tau_3^2).
\end{equation}
Since $\sigma_{X_3,X_4}=\tau_3$ and the packets $(X_3,X_4)$ form part of 
${\mathcal{B}''}$, we have $F_{\mathcal{B}''}(\tau_3)=\tau_3^\lambda$.
Thus the second line of \eqref{inns} together with \eqref{goodform} 
tell us that
$$F_{\mathcal{B}}(\tau_3)=f(\tau_3^2,Y_3)\tau_3^\lambda f(Y_3,\tau_3^2).$$
This concludes the proof. 

\vspace{.1cm}\noindent
(ii) This follows directly from Theorem \ref{Amovetheorem}, noting
that the maximal bracketing ${\mathcal{B}}$ is derived from the standard maximal 
bracketing ${\mathcal{B}_0}$ by three pairwise commutative A-moves: changing the
pair of brackets surrounding $1,2,3$ to one surrounding $3,4$, changing the pair of
brackets surrounding $1,2,3,4,5$ to one surrounding $5,6$, and changing the pair of
brackets surrounding $1,2,3,4,5,6,7$ to one surrounding $7,8$:
\begin{align*}
((((((&x_1,x_2),x_3),x_4),x_5),x_6),x_7),x_8\mapsto
(((((x_1,x_2),(x_3,x_4)),x_5),x_6),x_7),x_8\\
&\mapsto ((((x_1,x_2),(x_3,x_4)),(x_5,x_6)),x_7),x_8\mapsto
(((x_1,x_2),(x_3,x_4)),(x_5,x_6)),(x_7,x_8).
\end{align*}
\end{proof}

\section{Presentation of the pure mapping class group $\Gamma_{g,1}$}

\subsection{Wajnryb's presentation of $\Gone$}

In this background section, we recall the Wajnryb-Gervais presentation of 
$\Gamma_{g,1}$ in terms of the generators shown in Figure 1 above for the 
case $n=1$, so without any of the Dehn twists $g_i$ along curves $\gamma_i$.

\begin{thm}[Wajnryb \cite{W}] 
\label{wajnryb}
Let $g>1$. The following is a finite presentation for the mapping class group 
$\Gone$. The generators are the Dehn twists $a_1,a_2,\ldots,a_{2g}$ and $d$ 
along the corresponding simple loops shown in Figure 1;
see also Figure 2, with $d_2=d$. Define
\begin{equation}\label{wajnrybdefs}
\begin{cases}
y_k=a_{k-1}a_{k-2}\cdots a_2a_1^2a_2\cdots a_{k-1}a_k\hbox{ for }k=2,\ldots,2g+1\\
t_k=a_{2k}a_{2k-1}a_{2k+1}a_{2k}\hbox{ for }k=1,\ldots,g-1\\
d_3=(t_1^{-1}t_2^{-1}d_2t_2t_1)(t_2^{-1}d_2t_2)d_2a_1^{-1}a_3^{-1}a_5^{-1}\\
d'_2=y_5d_2y_5^{-1}\\
d'_3=y_7d_3y_7^{-1}.
\end{cases}
\end{equation}
Then the following set of relations between the above generators gives a presentation of $\Gone$:

\vspace{.1cm}\noindent
\rm{(A)} \ \ \ $a_ia_{i+1}a_i=a_{i+1}a_ia_{i+1}$ for $i=1,\ldots,2g-1$,

\vspace{.1cm}
\qquad $a_i$ commutes with $a_j$ for $|i-j|\ge 2$,

\vspace{.1cm}
\qquad $d_2a_4d=a_4d_2a_4$,

\vspace{.1cm}
\qquad $d_2$ commutes with $a_j$ for $j\ne 4$;

\vspace{.1cm}\noindent
\rm{(B)}  \ \ \ $\,d_2d'_2=(a_1a_2a_3)^4$; 

\vspace{.1cm}\noindent
\rm{(C)} \ \ \ $\,$(only needed for $g\ge 3$) $d_3a_6d_3=a_6d_3a_6$;

\vspace{.1cm}\noindent
\rm{(C$'$)} \ \ \  (only needed for $g\ge 4$) $d'_2$ commutes with $(t_2t_1t_3t_2)d_2(t_2t_1t_3t_2)^{-1}$. 

\end{thm}

\begin{proof}
This presentation can easily be deduced from Wajnryb's
presentation in \cite{W}.  In fact, the original relation (C) as written by Wajnryb is split here into 
two separate, simpler relations. This is possible because Wajnryb only uses his (C) to prove 
two further relations, numbered (18) and (30) in his proof, and we have chosen to use those two 
here instead of the original more complex form.  His (30) is our (C), and the relation in 
(18) which he proves using his (C) is our (C$'$). More precisely, in Wajnryb's notation, he states this
as the commutation of $e$ with $d_{34}$;  here $e$ is our $d'$ and $d_{34}$ is $(t_2t_1t_3t_2)d(t_2t_1t_3t_2)^{-1}$.
\end{proof}

\begin{figure}[H]
    \centering
    \includegraphics[width=0.6\textwidth]{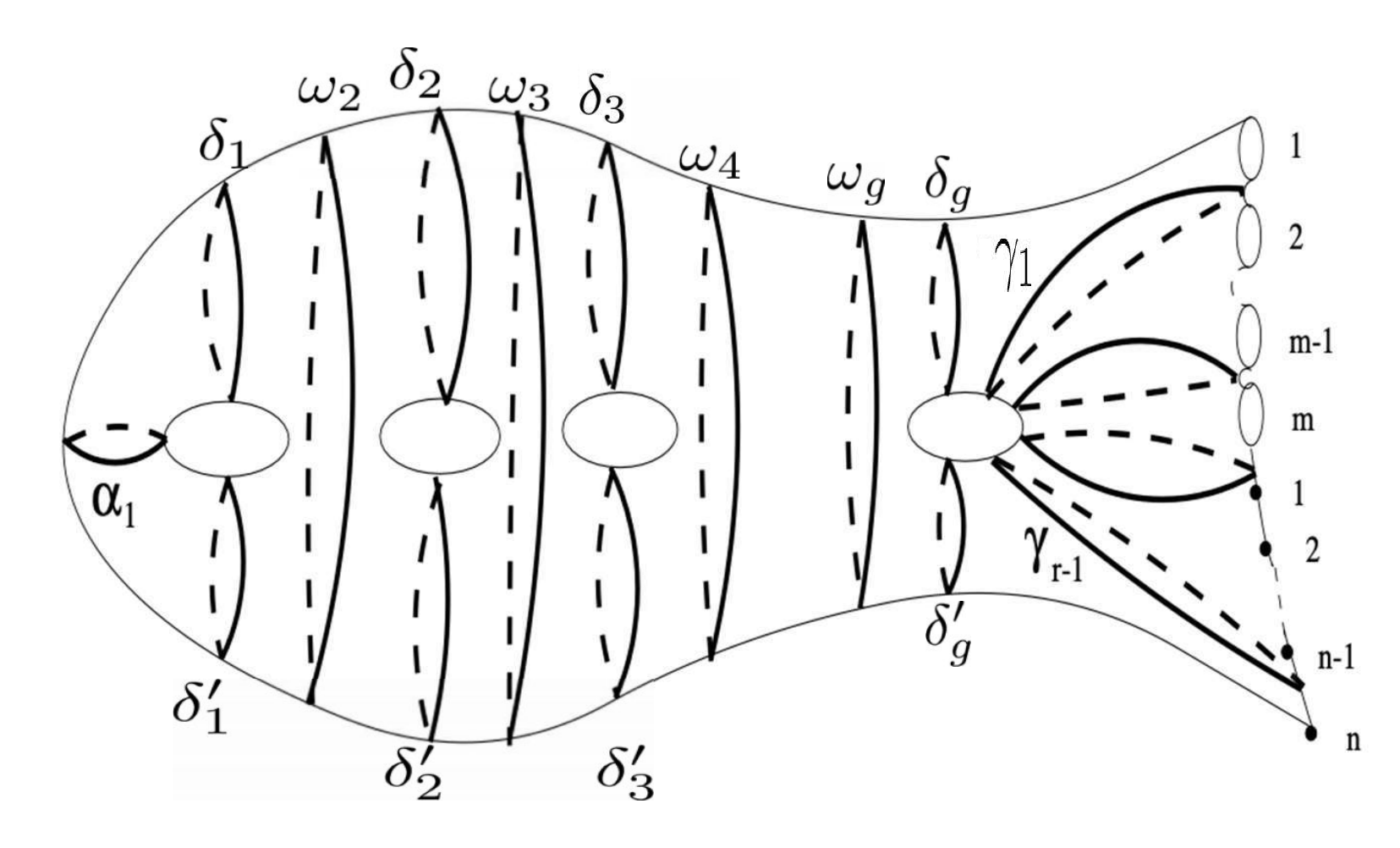}
    \caption{The loops $\delta_k$, $\delta'_k$ and $\omega_k$ (where $\alpha_1=\delta_1=\delta_1'$)}
    \label{fig:Fish2}
\end{figure}

\noindent {\bf Remark about notational conventions.} Wajnryb \cite{W} uses right twists
along a (simple) loop $\gamma$ in the mapping class
group that acts on the surface on the right: if we have twists $g,g'$ along 
loops $\gamma,\gamma'$ respectively, then the twist along the loop
$(\gamma')g$ is given by $g^{-1}g'g$.
In the following discussions, we switch this convention
(to follow most other literature) to the opposite, in which left twists along loops act
on the surface on the left, and a twist $g$ acting on a loop $\gamma'$ gives
a loop $g(\gamma')$ whose associated twist is $gg'g^{-1}$.
In particular, this changes the order of multiplication (or composition)
of twists: if $(g_1,\ldots,g_n)$ are right twists along a sequence of loops
$(\gamma_1,\dots,\gamma_n)$, then the composition of these twists acting
by $g_1,\dots, g_n$ in that order would be written as $g_1\cdots g_n$ in Wajnryb's 
convention (right action), but the same action
in our convention
 (right twists $g_i^{-1}$ acting on the left in the same order)  is written as 
 $g_n^{-1}\cdots g_1^{-1} =(g_1\cdots g_n)^{-1}$. This change from 
right to left action and the corresponding change in the notation of
composition has the effect of ensuring that Wajnryb's identities remain the
same in both notations.

\vspace{.3cm}
Let us give a pictorial explanation of some of the relations in Theorem \ref{wajnryb}. We note first that the relations in (A) follow simply from the fact 
that twists along disjoint loops commute,
and that the braid relation $aba=bab$ holds for twists $a,b$ along
any pair of loops $\alpha,\beta$ intersecting in a single point.

Next, we observe that the element $w_k=(a_1\ldots a_{k-1})^k$, image of the center of $B_k$,  is connected to the twist along the 
boundary on the minimal subsurface of $\Gone$ containing the braid loops $\alpha_1,\ldots,\alpha_{k-1}$. 
If $k$ is odd, then $w_k^2$ is the twist along the loop 
$\omega_{(k+1)/2}$ in Figure 2, and if $k$ is even, then
$w_k$ is the twist along the boundary of the surface cut out 
by 
$\delta_{k/2}$ and $\delta'_{k/2}$, 
i.e.~$w_k=d_{k/2}d'_{k/2}$. 
This explains relation (B). 

Relation (C), the braid relation $d_3a_6d_3=a_6d_3a_6$ is clear, but to 
include this relation in the presentation means that $d_3$ must be expressed 
in the generating set $d(=d_2),a_1,\ldots,a_{2g}$ as given in 
\eqref{wajnrybdefs}. For this, we first introduce Figure 3, which shows 
some definitions, loops and identities which will be useful later on.

\begin{figure}[H]
    \centering
\includegraphics[width=0.6\textwidth]{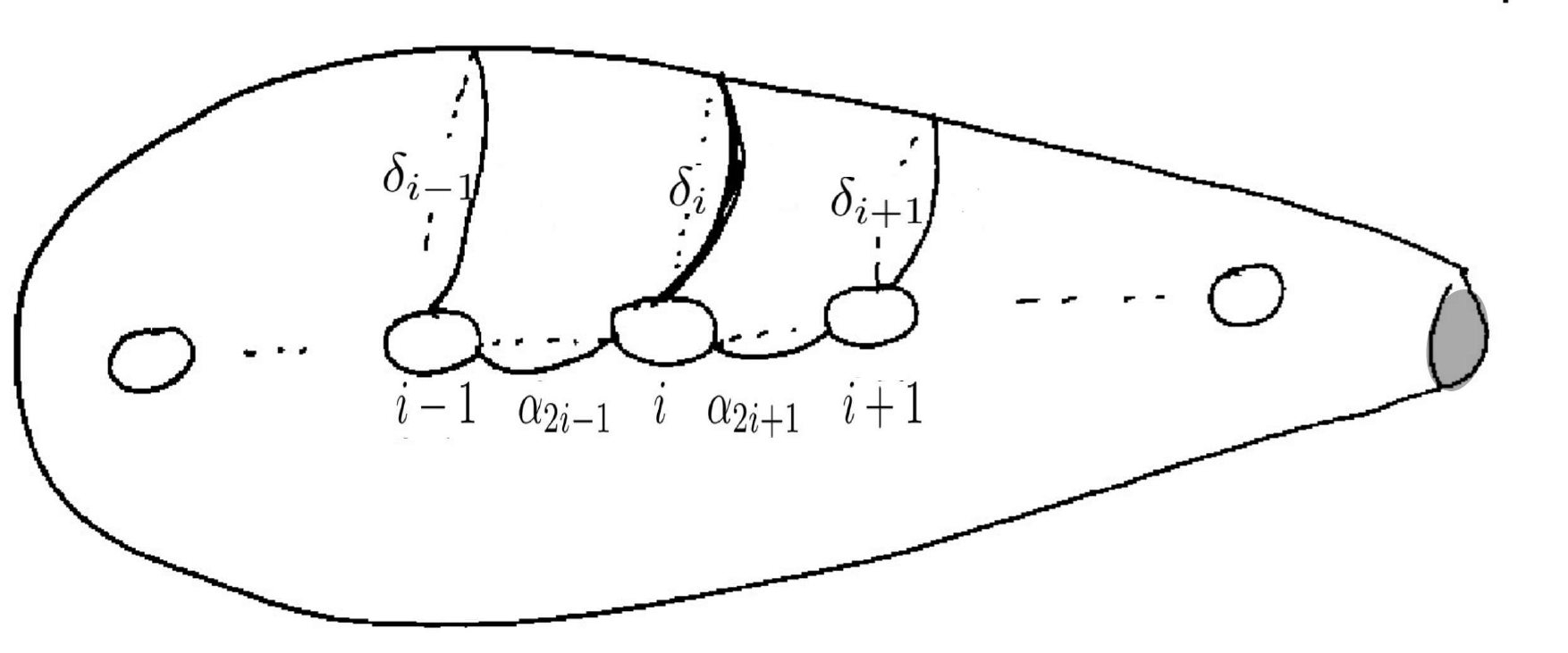}
\includegraphics[width=0.6\textwidth]{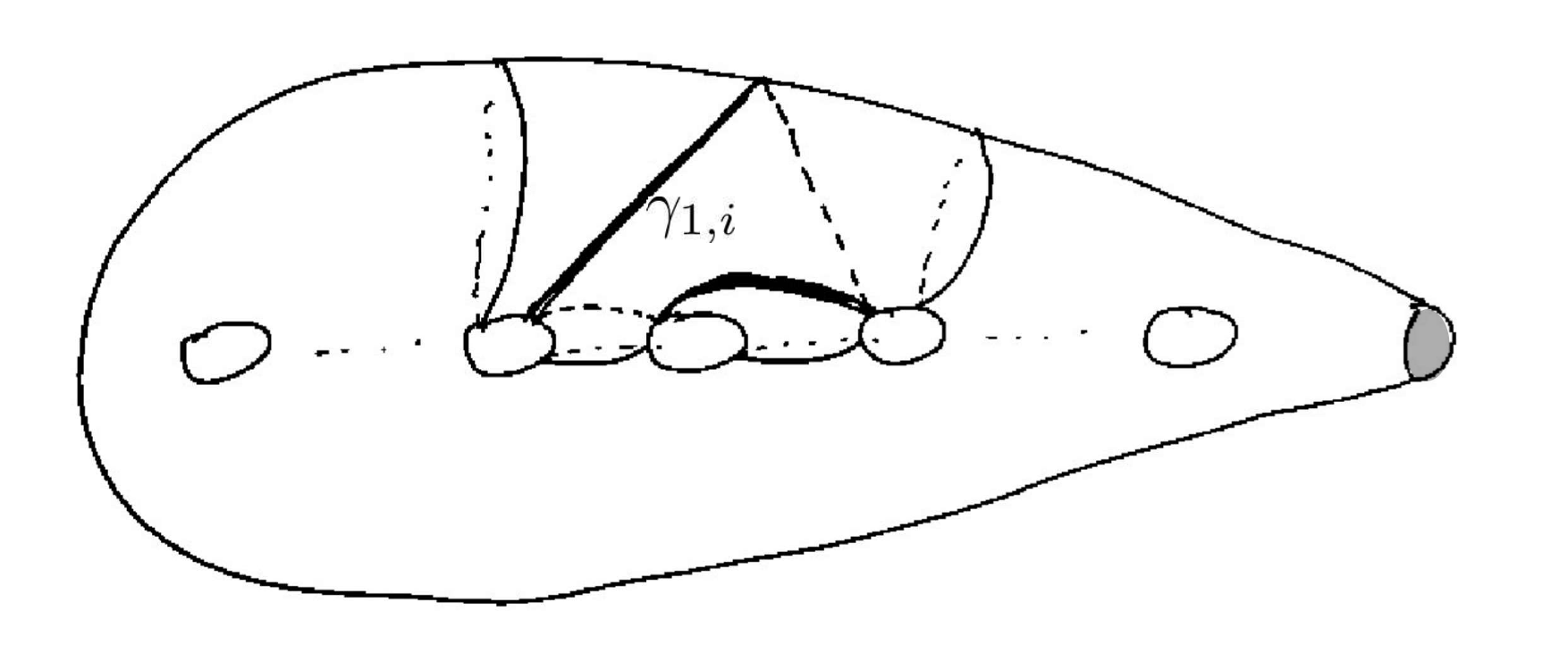}
\includegraphics[width=0.56\textwidth]{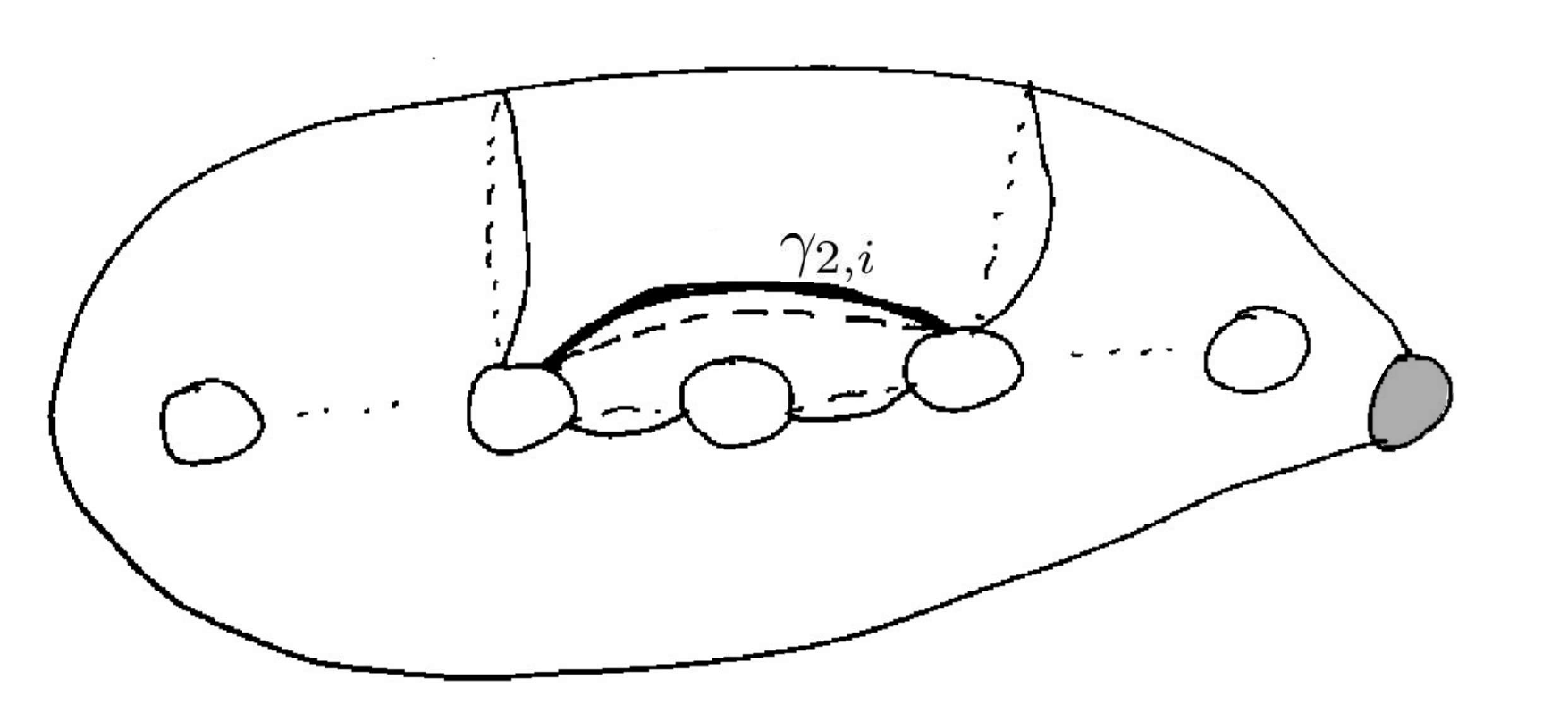}
    \caption{The bold loop $\gamma_{1,i}$ in the middle figure is given by 
$t_i^{-1}(\delta_i)$; its Dehn twist is $t_i^{-1}d_it_i$. The bold loop 
$\gamma_{2,i}$ in the lowest figure is given by $(t_it_{i-1}')^{-1}(\delta_i)$;
its Dehn twist is ${t}_{i-1}^{\prime -1}t_i^{-1}d_it_it'_{i-1}$.  }
    \label{fig:Fish3}
\end{figure}

\vspace{.2cm}
\begin{defn}\label{loops} {\rm We associate three loops to the $i$-th genus hole (see Figure 3).
The first is the loop $\delta_i$ depicted in the top diagram in Figure 3. 
The second is the loop $\gamma_{1,i}$ in the middle diagram in Figure 3,
defined by $\gamma_{1,i}=t_i^{-1}(\delta_i)$
where $t_i=a_{2i}a_{2i-1}a_{2i+1}a_{2i}$. Finally, the bottom loop $\gamma_{2,i}$ is given by $\gamma_{2,i}=t_{i-1}^{\prime -1}\cdot t_i^{-1}(\delta_i)$, where
we define $t'_{i-1}=a_{2i-2}d_{i-1}a_{2i-1}a_{2i-2}$. The three loops
shown in Figure 3 form a lantern, so the product of the three associated Dehn 
twists ${t}_{i-1}^{\prime -1}t_i^{-1}d_it_it'_{i-1}$, $t_i^{-1}d_it_i$ and 
$d_i$ is equal to the product of the boundary twists, given by
$d_{i-1}a_{2i-1}a_{2i+1}d_{i+1}$. } 
\end{defn}

\vspace{.3cm}
\noindent {\bf Remark 3.2.1.}
Let us briefly explain exactly why $t_i^{-1}(\delta_i)$ is the loop $\gamma_{1,i}$. It suffices to make
it clear in the case $i=2$. The diffeomorphisms $t_1a_1a_3=(a_2a_1a_3)^2$ and
$t_2a_3a_5=(a_4a_3a_5)^2$ act simultaneously on the subsurface $S_2$ cut out by $\alpha_1$, $\alpha_3$, $\alpha_5$ 
and $\delta_3$ and on the mirror subsurface $S'_2$ cut out by the same first three loops and $\delta'_3$; indeed,
in the case of $t_1$, we have $(t_1a_1a_3)^2=w_4=d_2d'_2$. However, restricted to the subsurface $S_2$,
the diffeomorphisms $t_1a_1a_3$ and $t_2a_3a_5$ act like half-twists along the loops $\delta_2$
and $\gamma_2$. In terms of braids, the strands $1,\ldots,6$ correspond to ramification points to the left 
and right of the three first genus holes in the $(g,1)$ surface under the rotation around the horizontal axis, 
and the half-twists $t_1a_1a_3$ and $t_2a_3a_5$ have the effect of braiding the corresponding four strands
in a half-twist (usually written $a_1a_2a_1a_3a_2a_1$ resp.~$a_3a_4a_3a_5a_4a_3$, but these 
are equal to $(a_2a_1a_3)^2$ resp.~$(a_4a_3a_5)^2$).

\vspace{.3cm}
Let us now return to the relation given by Wajnryb to express $d_3$ in
terms of the generators $d=d_2,a_1,\ldots,a_{2g}$:
$$d_3=(t_1^{-1}t_2^{-1}d_2t_2t_1)(t_2^{-1}d_2t_2)d_2a_1^{-1}a_3^{-1}a_5^{-1}.$$
The explanation of this complex formula becomes clear in view of 
Definition \ref{loops} with $i=2$.  In the case $i=2$, Figure 3 shows
the three loops $\delta_2$, $\gamma_{1,2}=t_2^{-1}(\delta_2)$ and
$\gamma_{2,2}=t_1^{\prime -1}t_2^{-1}(\delta_2)$. But in fact, although $t'_i\ne t_i$ in
general, it turns out that $t'_1=t_1$ because $\delta_1=\alpha_1$. Therefore
$\gamma_{2,2}=t_1^{-1}t_2^{-1}(\delta_2)$. The associated Dehn twists to these three
loops are then $t_1^{-1}t_2^{-1}d_2t_2t_1$, $t_2d_2t_2^{-1}$ and $d_2$,
and their product is a lantern product and therefore equal to the product
of the four boundary components:
$$t_1^{-1}t_2^{-1}d_2t_2t_1\cdot t_2d_2t_2^{-1}\cdot d_2=a_1a_3a_5d_3.$$
This explains the defining formula for $d_3$ given in 
\eqref{wajnrybdefs}.
 
\vspace{.2cm} The same figure is useful to understand relation (C$'$). 
Here, the braid $t_2t_1t_3t_2$ is best understood as
the flat braid crossing strands $1,2,3,4$ to the right above $5,6,7,8$. This
diffeomorphism, which only exists in genus $\ge 4$, preserves the top half
of the surface. More precisely, it acts on $\delta_2$ sending it to the 
curve $\gamma_{2,3}$ (see Figure~3), which is disjoint from 
$\delta'_2$, so that the two associated twists commute, 
which is the statement of (C$'$). 

\vspace{.3cm}

\section{The Dehn twist $d$ and the case of genus 2}

In this section we begin the study of the action of elements $F=(\lambda,f)\in \GT$
on $\GGone$, as proposed in Theorem \ref{thm11}, by proving that the proposed action of $F$ respects
all the relations in (A), as well as relation (B), which suffices to prove that $F\in\GT$ gives an 
automorphism of the mapping class group $\widehat\Gamma_{2,1}$. We first need some technical results
on homomorphisms of braid groups into $\GGone$, which are given in the first subsection. Observe that Wajnryb's presentation of $\Gone$ is also a topological
presentation for the profinite completion $\GGone$. Since everything concerning
$\GT$ takes place in the context of profinite groups, we only consider the
profinite completions of mapping class groups from here on.

\subsection{Homomorphisms from Artin braid groups into $\GGone$}

The first lemma concerns only the profinite Artin braid group $\B_{2k}$ and a particular quotient of it.
\begin{lem}\label{technical} Let $k\ge 1$. For $2\le i\le 2k$, let
$\eta_i=\sigma_{i-1}\cdots \sigma_1^2\cdots \sigma_{i-1}$ and
$\omega_i=(\sigma_1\cdots \sigma_{i-1})^i$ in $\B_{2k}$. Then
for $F\in \GT$, the automorphism of $\B_{2k}$ given by
\eqref{drinfeld} passes to the quotient $\B'_{2k}:=\B_{2k}/\langle
\omega_{2k-1}^{-1}\eta_{2k}\rangle$, and in the quotient we have
$$F(\sigma_{2k-1})=\sigma_{2k-1}^\lambda.$$
\end{lem}

\begin{proof}

 We will make use of the standard identities $\omega_i=\eta_2\cdots \eta_i$,
valid in both $\B_{2k}$ and $\B'_{2k}$. The standard automorphism $F$ of 
$\B_{2k}$ given in \eqref{drinfeld} satisfies $F(\eta_i)=\eta_i^\lambda$
and $F(\omega_i)=\omega_i^\lambda$ for $1\le i\le 2k$. Since all the
$\eta_i$ and the $\omega_i$ commute pairwise in $\B_{2k}$, the subgroup
$\langle \omega_{2k-1}^{-1}\eta_{2k}\rangle\subset \B_{2k}$ is preserved by
$F$, and therefore the action of $F$ on $\B_{2k}$ passes to the
quotient $\B'_{2k}$. It remains to show that $F(\sigma_{2k-1})=
\sigma_{2k-1}^\lambda$ in $\B'_{2k}$.

Consider the subgroup of $\B'_{2k}$ generated by $\eta_{2k-1}$ and $\sigma_{2k-1}$.
We can also take the two generators $x=\eta_{2k-1}$ and $y=\eta_{2k-1}^{-1}\sigma_{2k-1}^{-1}$ 
as generators of the same subgroup. Let us show that both 
$x^2$ and $y^2$ lie in the center of the subgroup $\langle x,y\rangle$
of $\B'_{2k}$. We first show that $x^2=\eta_{2k-1}^2$ commutes with
$\sigma_{2k-1}$. Indeed, $w_{2k}$ is in the center of $\B_{2k}$ and
therefore its image is also in the center of $\B'_{2k}$, so it commutes
with $\sigma_{2k-1}$ in $\B'_{2k}$, and in $\B'_{2k}$ we have
$$w_{2k}=w_{2k-1}\eta_{2k}=w_{2k-1}^2=
(\eta_2\cdots \eta_{2k-1})^2=\eta_2^2\eta_3^2\cdots \eta_{2k-1}^2,$$
where the last equality holds because the $\eta_i$ commute pairwise. Now,
the elements $\eta_2,\eta_3,\ldots,\eta_{2k-2}$ all commute with $\sigma_{2k-1}$,
$\eta_{2k-1}^2$ must also commute with $\sigma_{2k-1}$. Thus $x^2=\eta_{2k-1}^2$
commutes with both $x$ and $y$, so it lies in the center of 
$\langle x,y\rangle$. To see that $y^2$ is also in the center of 
$\langle x,y\rangle$, we write
$$y^2=\eta_{2k-1}^{-1}\sigma_{2k-1}^{-1}\eta_{2k-1}^{-1}\sigma_{2k-1}^{-1}=
\eta_{2k-1}^{-1}\eta_{2k}^{-1}.$$
Since all the $\eta_i$ commute pairwise, this shows that $y^2$ commutes with 
$x=\eta_{2k-1}$, and since $y^2$ commutes with $y$, we find that $y^2$
is also in the center of $\langle x,y\rangle$. 

This shows that every element $g\in \langle x,y\rangle$ can be written
$$g=y^{e_1}(xy)^\alpha x^{e_2}\cdot c$$
for some $c\in \langle x^2,y^2\rangle$, where $e_1,e_2\in \{0,1\}$.
If $g$ lies in the derived subgroup of $\langle x,y\rangle$, then the
total powers of $x$ and $y$ in $g$ must be equal to $0$, so we must have
\begin{equation}\label{alpha}
\begin{cases}
g=y(xy)^\alpha x\cdot c=(yx)^{\alpha+1}\cdot c&\hbox{if }\alpha\equiv 1\ mod\ 2\\
g=(xy)^\alpha\cdot c&\hbox{if }\alpha\equiv 0\ mod\ 2.
\end{cases}
\end{equation}
In other words, $g$ is either of the form $g=(xy)^\beta\cdot c$ or 
$g=(xy)^\beta\cdot c$ with $\beta\equiv 0$ mod 2.

Now, under the standard action of $F$ on $\B_{2k}$ given in \eqref{drinfeld},
we have
$$F(\sigma_{2k-1})=f(\sigma_{2k-1}^2,\eta_{2k-1})\sigma_{2k-1}^\lambda
f(\eta_{2k-1},\sigma_{2k-1}^2).$$
Thus, to show that $F(\sigma_{2k-1})=\sigma_{2k-1}^\lambda$ in the 
quotient group $\B'_{2k}$, it suffices to show that
$f(\eta_{2k-1},\sigma_{2k-1}^2)$ commutes with $\sigma_{2k-1}$. The element
$f(\eta_{2k-1},\sigma_{2k-1}^2)$ belongs to the derived subgroup of
$\langle x,y\rangle$, so by \eqref{alpha}, we have one of
$$f(\eta_{2k-1},\sigma_{2k-1}^2)=\begin{cases}
(\eta_{2k-1}^{-1}\sigma_{2k-1}^{-1}\eta_{2k-1})^\beta\cdot c\\
\sigma_{2k-1}^{-\beta}\cdot c,
\end{cases}
$$
with $\beta\equiv 0$ mod 2. In the second case, it is clear that
$f(\eta_{2k-1},\sigma_{2k-1})^2)$ commutes with $\sigma_{2k-1}$. In the
first case, we need to show that $\eta_{2k-1}^{-1}\sigma_{2k-1}^{-1}\eta_{2k-1}$
commutes with $\sigma_{2k-1}$. We saw above that $\eta_{2k-1}^2$ commutes with 
$\sigma_{2k-1}$ in $\B'_{2k}$. Using also the identity
$\eta_{2k}=\sigma_{2k-1}\eta_{2k-1}\sigma_{2k-1}$ and the pairwise commutation
of the $\eta_i$, we find that
\begin{align}
(\eta_{2k-1}^{-1}&\sigma_{2k-1}\eta_{2k-1})\sigma_{2k-1}(\eta_{2k-1}^{-1}\sigma_{2k-1}^{-1}\eta_{2k-1}
)=\eta_{2k-1}^{-1}\eta_{2k}\eta_{2k-1}^{-1}\sigma_{2k-1}^{-1}\eta_{2k-1}\notag\\
&=\eta_{2k}\eta_{2k-1}^{-2}\sigma_{2k-1}^{-1}\eta_{2k-1}=\eta_{2k}\sigma_{2k-1}^{-1}
\eta_{2k-1}^{-1}\notag\\
&=\eta_{2k}\sigma_{2k-1}^{-1}\eta_{2k-1}^{-1}\sigma_{2k-1}^{-1}\sigma_{2k-1}
=\eta_{2k}\eta_{2k}^{-1}\sigma_{2k-1}\notag\\
&=\sigma_{2k-1}.\notag
\end{align}
This completes the proof that $f(\eta_{2k-1},\sigma_{2k-1}^2)$ commutes
with $\sigma_{2k-1}$ in the quotient $\B'_{2k}$, so that
$F(\sigma_{2k-1})=\sigma_{2k-1}^\lambda$ in this quotient.
\end{proof}

\vspace{.2cm}
\noindent
{\bf Remark 4.1.1.}
{\rm Let $\B'_{2k}$ denote the quotient
$\B_{2k}/\langle w_{2k-1}^{-1}\eta_{2k}\rangle$. Observe that this quotient of 
$\B_{2k}$ is not the
same thing as the more familiar quotient by the subgroup
$\langle w_{2k-1}\eta_{2k}\rangle$. Indeed, we have the identity
$w_{2k}=w_{2k-1}\eta_{2k}=(\sigma_1\cdots \sigma_{2k-1})^{2k}$,
which generates the center of $\B_{2k}$, so the quotient by
$\langle w_{2k-1}\eta_{2k}\rangle$ is simply the quotient of
$\B_{2k}$ by its center, whereas the quotient by $\langle w_{2k-1}^{-1}
\eta_{2k}\rangle$ reflects the particular position of the Dehn twist $d_k$
on the topological surface.}

\vspace{.2cm}
The first two families of relations in (A) of Theorem \ref{wajnryb} show that
for all $2\le k\le 2g$ there is a homomorphism
\begin{align}\label{iota}
\iota:\B_k&\rightarrow \GGone\\
\sigma_i\ &\mapsto \ a_i.\notag
\end{align}

For $2\le k\le g$, let $\delta_k$ denote
the upper vertical loop from the $k$-th genus hole as in Fig. 2; 
in particular the twists $d_2=d$ and $d_3$ 
are the twists appearing in the Wajnryb presentation. In the next lemma,
we display a family of homomorphisms $\iota_k$ which differ from the
standard $\iota$ in that the final braid generator is mapped to the twist $d_k$.

\begin{lem}\label{iotas} Fix $g\ge 2$.  For each 
$k$ with $3\le k\le g+1$, there is a homomorphism of profinite groups 
\begin{equation}\label{iotak}
\iota_k:\B_{2k}\rightarrow \widehat\Gamma_{g,1}
\end{equation}
mapping $\sigma_i\mapsto a_i$ 
for $i=1,\ldots,2k-2$ and $\sigma_{2k-1}\mapsto d_{k-1}$.
The element $\omega_{2k-1}^{-1}\eta_{2k}$ lies in the kernel of $\iota_k$; 
in other words, the homomorphism $\iota_k$ factors through a homomorphism
$\iota'_k:\B'_{2k}\rightarrow \widehat\Gamma_{g,1}$. 
\end{lem}

\begin{proof}

The equality of $\iota_k(\omega_{2k-1})$ and
$\iota(\eta_{2k})$ in $\widehat\Gamma_{g,1}$ is a consequence of the following
standard relation in $\widehat\Gamma_{g,1}$ (already valid in
the discrete group $\Gamma_{g,1}$):
$$d_{k-1}d'_{k-1}=w_{2k-2} \ \ \ {\rm with}\ \ \ 
d'_{k-1}=y_{2k-1}d_{k-1}y_{2k-1}^{-1}$$
for $2\le k\le g+1$. Thus the elements
$d_{k-1}^{-1}w_{2k-2}$ and $y_{2k-1}d_{k-1}y_{2k-1}^{-1}$ are equal in $\GGone$.
We have $d=\iota_k(\sigma_{2k-1})$, and because $\sigma_{2k-1}$ does not
appear in $\omega_{2k-1}=(\sigma_1\cdots \sigma_{2k-2})^{2k-1}$ or in
$\eta_{2k-1}=\sigma_{2k-2}\cdots \sigma_1\cdot \sigma_1\cdots \sigma_{2k-2}$,
the element $w_{2k-1}\in \GGone$ is equal to the image under $\iota_k$ of
$\omega_{2k-1}\in \B_{2k}$, and similarly $y_{2k-1}$ is the image in $\GGone$
of $\eta_{2k-1}\in \B_{2k}$. Thus the two elements $d_{k-1}^{-1}w_{2k-2}$ and
$y_{2k-1}d_{k-1}y_{2k-1}^{-1}$, equal in $\GGone$,  are respectively the images 
of the elements $\sigma_{2k-1}^{-1}\omega_{2k-2}$ and $\eta_{2k-1}\sigma_{2k-1} 
\eta_{2k-1}^{-1}$ in $\B_{2k}$ under $\iota_k$. Therefore the elements
$\omega_{2k-2}\eta_{2k-1}$ and $\sigma_{2k-1}\eta_{2k-1}\sigma_{2k-1}$ also have
equal image in $\GGone$ under $\iota_k$. Noting that 
$\sigma_{2k-1}\eta_{2k-1}\sigma_{2k-1}=\eta_{2k}$ and
$\omega_{2k-2}\eta_{2k-1}=\omega_{2k-1}$ in $\B_{2k}$, we thus find that
$\omega_{2k-1}$ and $\eta_{2k}$ become equal under $\iota_k$, proving the 
lemma.
\end{proof}

\subsection{The action of $F$ on $\widehat\Gamma_{2,1}$}

The results in the preceding subsection will now allow us to prove that the
action of $F$ proposed in \eqref{Faction} respects all the relations in (A), as well as relation (B).
We begin by giving an easy proof of the first two families of relations in (A), based on
the standard homomorphism \eqref{iota}.

\begin{prop}\label{braidrels} 
The action of $F=(\lambda,f)\in \GT$ on
the Dehn twists given in \eqref{Faction} respects the braid relations
$a_ia_{i+1}a_i=a_{i+1}a_ia_{i+1}$ for $i=1,\ldots,2g-1$ and
$a_ia_j=a_ja_i$ for $|i-j|\ge 2$. Furthermore, the action of $F$ on the
elements $y_k=a_{k-1}\cdots a_1\cdot a_1\cdots a_{k-1}$ and $\omega_k=(\sigma_1\cdots\sigma_{k-1})^k$
is given by $F(y_k)=y_k^\lambda$ and $F(w_k)=w_k^\lambda$ for $k=2,\ldots,2g+1$.
\end{prop}

\begin{proof}
By comparing \eqref{drinfeld} with \eqref{Faction}, we
see that the homomorphism $\iota$ in \eqref{iota} satisfies
$\iota\circ F=F\circ \iota$ on the elements $\sigma_1,\ldots,\sigma_{2g}$.
Since the action of $F$ respects the braid relations between the $\sigma_i$,
$F\circ\iota$ respects the same relations in $\GGone$. This settles the first
two families of relations of type (A) in Theorem \ref{wajnryb}. As for the action on
$y_k$ and $w_k$, since $y_k=\iota(\eta_k)$ and $w_k=\iota(\omega_k)$, it follows 
directly from the fact that the action of $F$ on $\B_k$
is well-known to satisfy $F(\eta_k)=\eta_k^\lambda$ and $F(\omega_k)=\omega_k^\lambda$.
\end{proof}

\vspace{.3cm}
We now turn to the remaining relations in (A), and relation (B). Let $d=d_2$. 
The following results are direct corollaries of the preceding subsection.

\begin{cor}\label{Fiotacommute} Let $k=3$, let $g\ge 2$, and let 
$\iota_3:\B_6\rightarrow \widehat\Gamma_{2,1}$ be the homomorphism of \eqref{iotak}.
Then the homomorphisms $F\circ\iota_3$ and $\iota_3\circ F$ agree on the
generators $\sigma_1,\ldots,\sigma_5$ of $\B_6$, and are therefore
equal. In particular the 
action of $F$ on generators proposed in \eqref{Faction} respects the
relation $a_4da_4=da_4d$.
\end{cor}

\begin{proof}

 Since $\iota_3(\sigma_i)=a_i$ for $i=1,2,3,4$ and
$F(a_i)$ as given in \eqref{Faction} is the image under
$\iota_3$ of $F(\sigma_i)$ as given in \eqref{drinfeld}, we see that
$\iota_3\circ F$ and $F\circ \iota_3$ coincide on $\sigma_i$, $i=1,2,3,4$.  
For $\sigma_5$, we have $\iota_3(\sigma_5)=d_2=d$. By Lemma \ref{technical},
the action of $F$ passes to the quotient $\B'_6$, and in this quotient
we have $F(\sigma_5)=\sigma_5^\lambda$. Since $\iota_3$ factors
through $\iota'_3$ by Lemma \ref{iotas}, we find that
$\iota_3\bigl(F(\sigma_5)\bigr)=\iota'_3\bigl(F(\sigma_5)\bigr)
=\iota'_k(\sigma_5^\lambda)=d^\lambda$. 

To prove that the relation $a_4da_4=da_4d$ is respected by the action of $F$
on generators, we note that since $\iota_3$ is a group homomorphism and 
$F(\sigma_4)F(\sigma_5)F(\sigma_4)=F(\sigma_5)F(\sigma_4)F(\sigma_5)$ in $\B_6$,
and since we have $\iota_3\circ F=F\circ \iota_3$ on the $\sigma_i$, the 
image of this relation under $\iota_3$ gives $F(a_4)F(d)F(a_4)=F(d)F(a_4)F(d)$.
\end{proof}

\begin{cor}\label{relationB}
(i) The proposed action of $F$ respects relation (B) of
Theorem \ref{wajnryb}, namely $dd'=w_4$ where $d'=y_5dy_5^{-1}$.

\noindent (ii) The proposed action of $F$ respects the commutation of $d$ with $a_j$, $j\ne 4$.
\end{cor}
\begin{proof}
For (i), we write relation (B) as $dy_5dy_5^{-1}=w_4$, or equivalently,
$\iota_3(\sigma_5\eta_5\sigma_5\eta_5^{-1})=\iota_3(w_4)$,
or indeed $\iota_3(\sigma_5\eta_5\sigma_5)=\iota_3(w_4\eta_5)$.
Since $\sigma_5\eta_5\sigma_5=\eta_6$ and $w_4\eta_5=w_5$ this is equivalent to
$\iota_3(\eta_6)=\iota_3(w_5)$, so we must show that 
$$F\bigl(\iota_3(\eta_6)\bigr)=F\bigl(\iota_3(w_5)\bigr).$$ 
By the previous corollary, we have
$\iota_3\circ F=F\circ \iota_3$, so the desired equality can be written
$$\iota_3\bigl(F(\eta_6)\bigr)=\iota_3\bigl(F(w_5)\bigr).$$
Since $F(\eta_6)=\eta_6^\lambda$ and $F(w_5)=w_5^\lambda$ in $\B_n$, this equality holds,
proving that $F$ respects relation (B).

For (ii), we make use of the result from Proposition \ref{braidrels} which 
shows that $F(y_k)=y_k^\lambda$. Since $F(d)=d^\lambda$ by \eqref{Faction}
and $F(a_k)=f(a_k^2,y_k)a_kf(y_k,a_k^2)$, the commutation of $F(d)$ and $F(a_k)$
follows from the commutation of $d$ with both $a_k$ and $y_k$ for $k\ne 4$.
\end{proof}

\vspace{.3cm}
Thanks to the above and to the fact that by Theorem \ref{wajnryb}, the presentation of 
$\widehat\Gamma_{2,1}$ requires only relations (A) and (B), we have now obtained
the desired result in genus $g=2$.

\begin{cor}\label{genus2} The action of $F\in\GT$ proposed in \eqref{Faction}
extends to an automorphism of $\widehat\Gamma_{2,1}$.
\end{cor}

\section{Relation (C) and the case of genus 3}

From now to the end of the article, we let $g\ge 3$. Consider the elements $t_1$, $t_2$ 
and $t_3$ defined in Theorem \ref{wajnryb}. These elements are precisely the images of 
the elements $\tau_1,\tau_2$ and $\tau_3\in \B_8$ defined in \eqref{tausigmas}
under the homomorphism $\iota:\B_8\rightarrow \GGone$ mapping $\sigma_i\mapsto a_i$. 

Let ${\mathcal{B}}$, ${\mathcal{B}'}$ and ${\mathcal{B}''}$  be the maximal bracketings 
given in \eqref{threebracketings}. 
It is shown in Theorem \ref{doublestrandsubgroup} that 
\begin{equation}\label{newFaction}
F_{\mathcal{B}}={\rm inn}\bigl(f(y_3,a_3^2)f(y_5,a_5^2)f(y_7,a_7^2)\bigr)\circ F\\
\end{equation}

where by Lemma \ref{technical}, $F\in \GT$ is an automorphism of $\iota(\B_{2g})\subset \GGone$
acting on the $a_i$ by \eqref{drinfeld} (with $\sigma_i$ replaced by $a_i$), and thus 
the automorphism in \eqref{newFaction} of $\B_8$ is also an automorphism of
$\iota(\B_8)\subset\GGone$, acting as in Theorem \ref{Amovetheorem}. 
In particular, comparing with \eqref{drinfeld} we see that 
$F_{\mathcal{B}}(\sigma_i)=\sigma_i^\lambda$ for $i=1,3,5,7$. Also, we know by Theorem 
\ref{doublestrandsubgroup} that $F_{\mathcal{B}}$ acts on $\tau_1$, 
$\tau_2$ and $\tau_3$ analogously to \eqref{drinfeld}, i.e.
\begin{equation}\label{t1t2t3}
\begin{cases}
F_{\mathcal{B}}(t_1)=t_1^\lambda;
\\
F_{\mathcal{B}}(t_2)=f(t_2^2,t_1^2)\,t_2^\lambda \,f(t_1^2,t_2^2);
\\
F_{\mathcal{B}}(t_3)=f(t_3^2,t_2t_1^2t_2)\,t_3^\lambda \,f(t_2t_1^2t_2,t_3^2).
\end{cases}
\end{equation}
By a direct application of relation (II) of 
$\widehat{GT}$, we can write $F_{\mathcal{B}}(t_2)$ in the form
\begin{equation}\label{FBt2}
F_{\mathcal{B}}(t_2)=t_1^{-2m}t_2f(t_1^2,t_2^{-1}t_1^2t_2)t_1^{-2m}c^m
\end{equation}
where $c=(t_1t_2)^3=(t_1^2)\cdot (t_2^2)\cdot (t_2^{-1}t_1^2t_2)$ is the center of the subgroup 
$\langle t_1,t_2\rangle\subset \GGone$.

For the purpose of proving relation (C), we only need to study the action of $F_{\mathcal{B}}$
on the subgroup $\B_6$ of $\B_8$, i.e.~we only need to consider $t_1$ and $t_2$.
We can now turn to the proof that the action of $F$ respects relation (C), given by
$d_3a_6d_3=a_6d_3a_6$ with $d_3$ defined in Theorem \ref{wajnryb},
which is the Dehn twist along the curve $\delta_3$ shown in Figure 2. In order
to check this relation, we first need to compute the value of $F$ on the element $d_3$. 

\subsection{Three useful lemmas}

\begin{lem}[Haiku Lemma] {\it (i) Suppose $\{a,b\}$ and $\{A,B\}$ are two pairs of elements of 
a profinite group $G$ such that $a$ and $b$ both commute with $A$ and $B$.  
Then $f(a,b)f(A,B)=f(aA,bB)$.

\vspace{.1cm}\noindent
(ii) Suppose $a,b,g$ are elements of a profinite group $G$
such that $g$ commutes with $a$ and $b$, and suppose that $f\equiv a^rb^s$ in $G^{\rm ab}$.
Then $f(a,gb)=f(a,b)g^s$.

\vspace{.1cm}
\noindent (iii) Suppose that $a,b,g,f\in G$ are as in (ii) except that while $g$ still commutes with
$a$, we now assume that $gb=bg^{-1}$. Then $f(a,gb)=f(a,b)g^e$ where $e=0$ if $s$ is even, 
$e=-1$ if $s$ is odd.}
\end{lem}

\begin{proof}

(i) is easy to show when $f$ is a word in a finite quotient of
$G$; then we can take the limit from finite words to pro-words. (ii) follows directly from (i)
with $a=1$ and $b=g$. We prove (iii) again for all words $f$ in finite quotients of $G$, but
now we do it by induction. The base cases are words containing 1 or 2 occurrences of $b$ or $b^{-1}$.
If $f(a,b)$ is a word containing a single $b$ or $b^{-1}$, then $f(a,gb)=f(a,b)g^{-1}$,
If $f$ contains two occurrences
of $b$ or $b^{-1}$, then $f(a,gb)=f(a,b)$ as passing the $g$'s between the two occurrences of $b$
causes them to cancel. We complete the proof by induction. Let $w$ be a word having $m$ occurrences
of $b$ or $b^{-1}$, so that $w(a,gb)=w(a,b)g^e$, and let $w'$ be a word with a single occurrence
of $b$ or $b^{-1}$; let $f=ww'$. Then $f(a,gb)=w(a,gb)w'(a,b)=w(a,b)g^ew'(a,b)g^{-1}$.
If $e=0$ then the result is proven. If $e=-1$, then $f(a,gb)=w(a,b)g^{-1}w'(a,b)g^{-1}$ and
the middle $g^{-1}$ changes to $g$ when it passes from left to right of the single $b$ in $w'$,
so we get $f(a,gb)=f(a,b)$. 
\end{proof}

\begin{lem}\label{useful0}
(i) For $i=2,3,4$, $y_i$ commutes with $d$ and $d'$, and for $i=5,\ldots,g+1$, $y_i^2$ commutes
with $d$ and $d'$.

\vspace{.1cm}\noindent
(ii) For $i=5,\ldots,g+1$, we have
\begin{equation}\label{useful1}
d'=y_idy_i^{-1}=y_i^{-1}dy_i.
\end{equation}
\end{lem}

\begin{proof}

 (i) The three elements $y_2,y_3$ and $y_4$ obviously commute with
$d$ and $d'$ since they are products of the twists $a_1$, $a_2$ and $a_3$ whose underlying
loops are disjoint from $\delta$ and $\delta'$. Let $i>4$. 
We use the identity $w_{i-1}y_i=w_i$, and the fact that the
elements $w_i^2$ are Dehn twists along the loops $\omega_i$ (see Figure 2) which are 
all disjoint from $d$ and $d'$.  The $w_i$ and the $y_j$ all commute between each other
and with each other (recall that $w_i=y_2\cdots y_i$), so we have
$w_i^2=(w_{i-1}y_i)^2=w_{i-1}^2y_i^2$. Thus since $w_{i-1}^2$ and $w_i^2$ both 
commutes with $d$ and $d'$, $y_i^2$ must as well. 

\vspace{.1cm} \noindent (ii) For $i=5$, we have $d'=y_5dy_5^{-1}$ by definition. Let $i>5$
and write $y_i=(a_{i-1}\cdots a_5)y_5(a_5\cdots a_{i-1})$.
Then
$$y_idy_i^{-1}=(a_{i-1}\cdots a_5)y_5(a_5\cdots a_{i-1})d(a_{i-1}\cdots a_5)^{-1}y_5^{-1}(a_5\cdots a_{i-1})^{-1}.$$
Because $d$ commutes with all $a_i$ for $i\ne 4$, this is equal to
$$(a_{i-1}\cdots a_5)y_5\,d\,y_5^{-1}(a_5\cdots a_{i-1})^{-1}
=(a_{i-1}\cdots a_5)\,d'\, (a_5\cdots a_{i-1})^{-1},$$
but since $d'$ also commutes with all $a_j$ for $j\ne 4$, we obtain 
$d'=y_idy_i^{-1}$. Finally, since $y_i^2$ commutes with $d'$ by (i), we can conjugate this equality
by $y_i^{-2}$ to find that $d'$ is also equal to $y_i^{-1}dy_i$, completing the proof.
\end{proof}

\begin{lem}\label{FBddprime} $F_{\mathcal{B}}(d)=d^\lambda$ and $F_{\mathcal{B}}(d')={d'}^\lambda$.
\end{lem}

\begin{proof}

By \eqref{Faction}, we know that $F(d)=d^\lambda$, and by 
Proposition \ref{braidrels}, we know that $F(y_i)=y_i^\lambda$ for $i=2,\ldots,g+1$.
By definition, $d'=y_5dy_5^{-1}$, so we also have
$$F(d')=F(y_5)F(d)F(y_5^{-1})=y_5^\lambda F(d) y_5^{-1}=y_5 d^\lambda y_5^{-1}={d'}^\lambda.$$
Thus
\begin{equation*}
\begin{cases}
F_{\mathcal{B}}(d)=f(a_7^2,y_7)f(a_5^2,y_5)f(a_3^2,y_3)\,d^\lambda\, f(y_3,a_3^2)f(y_5,a_5^2) f(y_7,a_7^2)\\
F_{\mathcal{B}}(d')=f(a_7^2,y_7)f(a_5^2,y_5)f(a_3^2,y_3)\,{d'}^\lambda\, f(y_3,a_3^2)f(y_5,a_5^2) f(y_7,a_7^2).
\end{cases}
\end{equation*}
The twists $d$ and $d'$ commutes with $a_3$, $a_5$ and $a_7$ since the underlying loops are disjoint, and they also commute with $y_3$ as observed in
Lemma \ref{useful0} (i), so
the inner term $f(y_3,a_3^2)$ drops out and we have
\begin{equation*}
\begin{cases}
F_{\mathcal{B}}(d)=f(a_7^2,y_7)f(a_5^2,y_5)\,d^\lambda\, f(y_5,a_5^2) f(y_7,a_7^2)\\
F_{\mathcal{B}}(d')=f(a_7^2,y_7)f(a_5^2,y_5)\,{d'}^\lambda\, f(y_5,a_5^2) f(y_7,a_7^2).
\end{cases}
\end{equation*}
Rewriting the first equality of \eqref{useful1} as
\begin{equation}\label{useful2}
{d'}^{-1}y_i=y_id^{-1} \ \ \ {\rm and}\ \ \ d'y_i=y_id,
\end{equation} 
and using the commutation of $y_i^2$ with $d$, we find that for $i>4$, $y_i$ anti-commutes with
$d^{-1}d'$, i.e.
\begin{equation}\label{useful3}
y_id^{-1}d'={d'}^{-1}y_id'={d'}^{-1}y_i^2dy_i^{-1}={d'}^{-1}dy_i=d{d'}^{-1}y_i,
\end{equation}
or equivalently, $d^{-1}d'y_i=y_id{d'}^{-1}$.
By (iii) of the Haiku Lemma, we deduce from \eqref{useful3} that for $i=5,7$, we have
\begin{equation}\label{useful4}
f(a_i^2,y_i\cdot d^{-1}d'y_i)=f(a_i^2,y_i\cdot d^{-1}d')=f(a_i^2,y_i),
\end{equation}
since $f$ lies in the derived subgroup of $\widehat{F}_2$.  So we see that for $i=5,7$ we have
\begin{align*}
f(a_i^2,y_i)\,d\,f(y_i,a_i^2)&=d\,f(a_i^2,d^{-1}y_id)f(y_i,a_i^2)\\
&=df(a_i^2,d^{-1}d'y_i)f(y_i,a_i^2)\hbox{ by }\eqref{useful2}\\
&=df(a_i^2,y_i)f(y_i,a_i^2)\hbox{ by }\eqref{useful4}\\
&=d.
\end{align*}
and
\begin{align*}
f(a_i^2,y_i)\,d'\,f(y_i,a_i^2)&=d'\,f(a_i^2,{d'}^{-1}y_id')f(y_i,a_i^2)\\
&=d'f(a_i^2,y_id^{-1}d')f(y_i,a_i^2)\hbox{ by }\eqref{useful2}\\
&=d'f(a_i^2,y_i)f(y_i,a_i^2)\hbox{ by }\eqref{useful4}\\
&={d'}.
\end{align*}
Thus in particular both $f(a_5^2,y_5)$ and $f(a_7^2,y_7)$ commute with $d$ and $d'$, proving
the Lemma.
\end{proof}

\subsection{The lantern relation and relation (II)}

The usual ``full'' relation (II) on elements $(\lambda,f)$ of $\GT$ is written as
$$f(x,y)x^mf(z,x)z^mf(y,z)y^m=1$$ 
where $xyz=1$ and $m=(\lambda-1)/2$. This can be generalized to the case where the element 
$xyz$ commutes with $x$, $y$ and $z$, in which case the relation becomes
\begin{equation}\label{relIIversion0}
f(x,y)x^m f(z,x)z^m f(y,z)y^m=(xyz)^m,
\end{equation}
as can be easily seen by passing to the abelianization and recalling that $f$ is a (profinite)
commutator (see also [NS]). Let $S$ be a topological surface of genus $0$ with four boundary components,
and $\chi$, $\upsilon$ and $\zeta$ represent the three standard loops on $S$ surrounding the
boundary components $(1,2)$,$(1,3)$ and $(2,3)$ respectively. Then letting $x$, $y$ and $z$
denote the Dehn twists along $\chi$, $\upsilon$ and $\zeta$ respectively, 
and writing $b_1,\ldots,b_4$ for the Dehn twists along the four
boundary components, the ``lantern relation'' on Dehn twists on $S$ is given by $xyz=b_1b_2b_3b_4$,
and the standard action of $F$ on the (profinite) mapping class group of $F$ results in
the version of relation (II) written as
\begin{equation}\label{relIIversion1}
f(x,y)x^m f(z,x)z^m f(y,z)y^m=(b_1b_2b_3b_4)^m.
\end{equation}
Observing that if $(\lambda,f)$ lies in $\GT$ then so does $(-\lambda,f)$, the relation 
\eqref{relIIversion1} holding for $m=(\lambda-1)/2$ also holds for $(-\lambda-1)/2=-m-1$, giving
\begin{equation}\label{relIIversion2}
f(x,y)x^{-m-1} f(z,x)z^{-m-1} f(y,z)y^{-m-1}=(b_1b_2b_3b_4)^{-m-1},
\end{equation}
which can be inverted (and the terms cyclically permuted) to give the identity
\begin{equation}\label{relIIversion3}
f(x,z)x^{m+1}f(y,x)y^{m+1}f(z,y)z^{m+1}=(b_1b_2b_3b_4)^{m+1}.
\end{equation}

We will now apply this to the situation where the subsurface $S$ of type $(0,4)$ is cut out of
the surface of type $(g,1)$ in Figure 1 by the loops $a_1$, $a_3$, $a_5$ and $d_3$. Consider
the Dehn twists
\begin{equation}\label{g1g2def}
g_1=t_2^{-1}dt_2\ \ \ {\rm and}\ \ \ g_2=t_1^{-1}t_2^{-1}dt_2t_1,
\end{equation}
along the loops pictured in Figure 3 above. 
Together with $d_2$, these three diffeomorphisms are twists along the three
standard loops $\delta_2$, $\gamma_{1,2}$ and $\gamma_{1,3}$ on the genus $0$ 
subsurface $S$ cut out by the four boundary components $\alpha_1$, 
$\alpha_3$, $\alpha_5$ and $\delta_3$ (see Figure~3). Thus they satisfy the lantern relation
\begin{equation}\label{lantern}
g_2g_1d=a_1a_3a_5d_3,
\end{equation}
which is how the complicated-looking defining relation $d_3=g_2g_1da_1^{-1}a_3^{-1}a_5^{-1}$
arises, since $d_3$ is a priori not in the generating set of $\GGone$. Since this is a lantern
relation, we can apply \eqref{relIIversion0}
with $x=d$, $y=g_2$, $z=g_1$ to obtain
\begin{equation}\label{relIIversion4}
f(d,g_2)d^mf(g_1,d)g_1^mf(g_2,g_1)g_2^m=(g_2g_1d)^m,
\end{equation}
and we can also apply \eqref{relIIversion3} taking
$x=g_2$, $y=g_1$, $z=d$, $b_1, b_2, b_3, b_4=a_1,a_3,a_5,d_3$  and using \eqref{lantern} to write
\begin{equation}\label{relIIversion5}
f(g_2,d)g_2^{m+1}f(g_1,g_2)g_1^{m+1}f(d,g_1)d^{m+1}=(g_2g_1d)^{m+1}.
\end{equation}
We will use these last two equalities in the next subsection to compute $F(d_3)$.

\subsection{Computation of $F(d_3)$}\label{CompFd3} To compute $F(d_3)$, we will compute the values 
$F$ takes on the elements $g_1$, $g_2$.  We start by computing the values of $F_{\mathcal{B}}$
on these elements. Using $F_{\mathcal{B}}(d)=d^\lambda$, $F_{\mathcal{B}}(t_1)=t_1^\lambda$ 
(by \eqref{t1t2t3} and \eqref{g1g2def}), as well as Lemma \ref{FBddprime}, we find that
\begin{equation}\label{g1g2}
\begin{cases}
F_{\mathcal{B}}(g_1)=t_1^{2m}f(t_2^{-1}t_1^2t_2,t_1^2)\,g_1^\lambda \,
f(t_1^2,t_2^{-1}t_1^2t_2)t_1^{-2m}\\
F_{\mathcal{B}}(g_2)=f(t_2^2,t_1^2)\,g_2^\lambda\, f(t_1^2,t_2^2)\\
F_{\mathcal{B}}(d)=d^\lambda.
\end{cases}
\end{equation}
Now, $a_1^2$ and $a_3^2$ commute with $t_1^2$, $t_2^{-1}t_1^2t_2$, $t_2^2$ and $d$, so
we can replace $t_1^2$ by $a_1^2a_3^2t_1^2$ in all the $f$'s of \eqref{g1g2}; furthermore
conjugating by $t_2$ exchanges $a_1^2$ and $a_3^2$, which implies that the product
$a_1^2a_3^2$ commutes with $g_1$,
so we can also conjugate the
first line of \eqref{g1g2} by $a_1^2a_3^2$ without changing its value. For the second line,
we further note that after changing $t_1^2$ to $a_1^2a_3^2t_1^2$ in the argument of the $f$'s,
we can also change $t_2^2$ to $a_3^2a_5^2t_2^2$ in the arguments of $f$, since
$a_3^2$ and $a_5^2$ both commute with $a_1^2a_3^2t_1^2$.  (Observe that $a_1^2a_3^2t_1^2$ is 
nothing other than the full twist $T_{1234}$, while $a_3^2a_5^2t_2^2$ is the full twist $T_{3456}$.) 
Doing all this, we rewrite \eqref{g1g2} as
\begin{equation}\label{g1g2bis}
\begin{cases}
F_{\mathcal{B}}(g_1)=a_1^2a_3^2t_1^{2m}f(t_2^{-1}a_1^2a_3^2t_1^2t_2,a_1^2a_3^2t_1^2)\,g_1^\lambda \,
f(t_1^2,t_2^{-1}a_1^2a_3^2t_1^2t_2)t_1^{-2m}a_1^{-2}a_3^{-2}\\
F_{\mathcal{B}}(g_2)=f(a_3^2a_5^2t_2^2,a_1^2a_3^2t_1^2)\,g_2^\lambda\, f(a_1^2a_3^2t_1^2,a_3^2a_5^2t_2^2).
\end{cases}
\end{equation}
Next, set
$$g'_1=t_2^{-1}d't_2, \quad g'_2=t_1^{-1}t_2^{-1}d't_2t_1.$$ 
Then $g_1g'_1=t_2^{-1}dd't_2$ and $g_2g'_2=t_1^{-1}t_2^{-1}dd't_2t_1$, so since 
$dd'=(a_1a_2a_3)^4$ in $\GGone$ by relation (B), these braids lie in $\B_6$, and direct braiding
allows us to easily check the following identities:
\begin{equation}\label{a1a3}
\begin{cases}
a_1^2a_3^2t_1^2=a_1^2a_3^2(a_2a_1a_3a_2)^2=(a_1a_2a_3)^4=w_4=dd'\\
t_2^{-1}a_1^2a_3^2t_1^2t_2=g_1g'_1\\       
a_3^2a_5^2t_2^2=g_2g'_2.  
\end{cases}
\end{equation}
Using these, we transform \eqref{g1g2bis} into
\begin{equation}\label{g1g2ter}
\begin{cases}
F_{\mathcal{B}}(g_1)=(dd')^m f(g_1g'_1,dd')\,g_1^\lambda\, f(dd',g_1g'_1)(dd')^{-m}\\
F_{\mathcal{B}}(g_2)=f(g_2g'_2,dd')\,g_2^\lambda\, f(dd',g_2g'_2). 
\end{cases}
\end{equation}
These expressions can be simplified by applying the Haiku Lemma.
Since the pair $\{d,g_1\}$ commutes with $\{d',g'_1\}$ and the pair $\{d,g_2\}$ 
commutes with the pair $\{d',g'_2\}$, we use (i) of the Haiku Lemma, and find that
$$F_{\mathcal{B}}(g_1)=(dd')^mf(g_1,d)f(g'_1,d')g_1^\lambda f(d',g'_1)f(d,g_1)(dd')^{-m} =d^m f(g_1,d)g_1^\lambda f(d,g_1)d^{-m}$$
and
$$F_{\mathcal{B}}(g_2)=f(g_2,d)f(g'_2,d')g_2^\lambda f(d',g'_2)f(d,g_2)=f(g_2,d)g_2^\lambda f(d,g_2).$$
These allow us to compute 
$F_{\mathcal{B}}(g_2)F_{\mathcal{B}}(g_1)F_{\mathcal{B}}(d)$ as follows:
\begin{align*}
F_{\mathcal{B}}(g_2)F_{\mathcal{B}}(g_1)F_{\mathcal{B}}(d)&=f(g_2,d)g_2^\lambda f(d,g_2)d^mf(g_1,d)g_1^\lambda
f(d,g_1)d^{-m}d^\lambda\\
&=f(g_2,d)g_2^\lambda g_2^{-m}f(g_1,g_2)(g_2g_1d)^m g_1^{-m}g_1^\lambda f(d,g_1)d^{m+1}
\ \ \hbox{by}\ \eqref{relIIversion4}\\
&=f(g_2,d)g_2^{m+1}f(g_1,g_2)g_1^{m+1}f(d,g_1)d^{m+1}(g_2g_1d)^m\\
&=(g_2g_1d)^{m+1}(g_2g_1d)^m\hfill\qquad \hbox{by}\ \eqref{relIIversion5}\\
&=(g_2g_1d)^\lambda.
\end{align*}
Since $g_2g_1d=a_1a_3a_5d_3$, this shows that $F_{\mathcal{B}}$ respects the lantern relation 
 in $\GGone$. In particular:
\begin{align*}
F_{\mathcal{B}}(d_3)&=F_{\mathcal{B}}(g_2)F_{\mathcal{B}}(g_1)F_{\mathcal{B}}(d)F_{\mathcal{B}}(a_1^{-1})F_{\mathcal{B}}(a_3^{-1})
F_{\mathcal{B}}(a_5^{-1})\\
&=(g_2g_1d)^\lambda a_1^{-\lambda}a_3^{-\lambda}a_5^{-\lambda}\\
&=d_3^\lambda.
\end{align*}
To conclude, we recall from \eqref{newFaction} that 
$$F_{\mathcal{B}}={\rm inn}\Bigl(f(y_3,a_3^2)f(y_5,a_5^2)f(y_7,a_7^2)\Bigr)\circ F.$$
Since $d_3$ commutes with $y_3,y_5,y_7,a_3,a_5$ and $a_7$, we find that $F(d_3)=d_3^\lambda$.

\subsection{Proof of relation (C)}

We can now complete the proof that $F$ respects relation (C).
Recall that the homomorphism $\iota_4:\B_8\rightarrow \GGone$ maps $\sigma_i\mapsto a_i$
for $i=1,\ldots,6$ and $\sigma_7\mapsto d_3$. By Lemmas \ref{technical} and \ref{iotas}, 
the map $\iota_4$ factors through the quotient $\B_8\rightarrow \B'_8$ of $\B_8$ by the
subgroup $\langle w_7^{-1}\eta_8\rangle$, and writing $\iota'_4$ for the homomorphism
$\B'_8\rightarrow \GGone$, we have $\iota'_4\circ F=F\circ \iota'_4$ on $\sigma_1,\ldots,\sigma_6$.
Since the same Lemmas show that the action of $F$ on $\B_8$ passes to an action of
$F$ on $\B'_8$ such that $F(\sigma_7)=\sigma_7^\lambda$ in $\B'_8$, and since we just saw
in the previous subsection that $F(d_3)=d_3^\lambda$, we find that $\iota'_4\circ F$ coincides
with $F\circ \iota'_4$ on all of $\B'_8$. Thus $F$ is an automorphism of the image of
$\B'_8$ in $\GGone$ and therefore $F$ respects the relation
$$\iota'_4(\sigma_6) \iota'_4(\sigma_7) \iota'_4(\sigma_6)
=\iota'_4(\sigma_7) \iota'_4(\sigma_6) \iota'_4(\sigma_7)$$
in $\GGone$, which is precisely the relation $a_6d_3a_6=d_3a_6d_3$.

This proves that the proposed action of $F$ on $\GGone$ respects relation (C). Since we also
have relations (A) and (B), proved in the previous sections, and since relation (C') only
concerns the case $g\ge 4$, we have proved

\vspace{.3cm}
\begin{cor}\label{genus3} The action of $F\in\GT$ proposed in \eqref{Faction}
extends to an automorphism of $\widehat\Gamma_{3,1}$.
\end{cor}

\section{Relation (C$'$) and the general case $g\ge 4$}
In this section we return to the notation $d$ instead of $d_2$, as the other $d_i$ are not used here.
Recall that $d'=y_5dy_5^{-1}$, and
that relation (C$'$) says that $d'$ commutes with $(t_2t_1t_3t_2)d(t_2t_1t_3t_2)^{-1}$. 
We will actually use the automorphism
$$F_{\mathcal{B}}={\rm inn}\bigl(f(y_3,a_3^2)f(y_5,a_5^2)f(y_7,a_7^2)\bigr)\circ F$$ 
of $\B_8$ introduced in the previous section, which satisfies
$F_{\mathcal{B}}(a_i)=a_i^\lambda$ for $i=1,3,5,7$, and show that $F_{\mathcal{B}}$ respects
relation (C$'$), which automatically implies the same result for $F$ since they differ
by an inner automorphism.

Recall the maximal bracketing
$${\mathcal{B}''}=((x_1,x_2),(x_3,x_4)),((x_5,x_6),(x_7,x_8))$$
from \eqref{threebracketings}. As noted there, ${\mathcal{B}''}$ differs from
${\mathcal{B}}$ by the single A-move changing the bracket surrounding $1,\ldots,6$ to the
bracket surrounding $5,6,7,8$, so we have
\begin{equation}\label{BB}
F_{\mathcal{B}''}={\rm inn}\,f(T_{123456},T_{5678})\circ F_{\mathcal{B}}.
\end{equation}
In $\B_8$, we have
$$T_{123456}=\omega_6\ \ {\rm and}\ \ T_{5678}=\sigma_5^2\sigma_7^2\tau_3^2.$$
Thus under the homomorphism $\iota:\B_8\rightarrow \GGone$, 
$T_{123456}=\omega_6\mapsto w_6$ and $T_{5678}=\sigma_5^2\sigma_7^2\tau_3^2\mapsto 
a_5^2a_7^2t_3^2$. Noting that $a_5$ and $a_7$ both commute with $t_3^2$, we find that
$$f(T_{123456},T_{5678})=f(\omega_6,\sigma_5^2\sigma_7^2\tau_3^2)\mapsto
f(w_6,a_5^2a_7^2t_3^2)=f(w_6,t_3^2).$$
Thus \eqref{BB} translates in $\GGone$ to
\begin{equation}\label{BB2}
F_{\mathcal{B}''}={\rm inn}\,f(w_6,t_3^2)\circ F_{\mathcal{B}}.
\end{equation}
But since by Theorem \ref{Amovetheorem} we have
\begin{equation*}
\begin{cases}
F_{\mathcal{B}''}(T_{1234})=T_{1234}^\lambda\\
F_{\mathcal{B}''}(T_{5678})=T_{5678}^\lambda\\
F_{\mathcal{B}''}(T_{12345678})=T_{12345678}^\lambda,
\end{cases}
\end{equation*}
we see that 
$$F_{\mathcal{B}''}(t_2t_1t_3t_2)=(t_2t_1t_3t_2)^\lambda.$$
Thus by \eqref{BB2}, we have
\begin{equation}\label{BB3}
F_{\mathcal{B}}(t_2t_1t_3t_2)=f(t_3^2,w_6)(t_2t_1t_3t_2)^\lambda f(w_6,t_3^2).
\end{equation}
We use this to compute $F_{\mathcal{B}}\bigl((t_2t_1t_3t_2)^{-1}d(t_2t_1t_3t_2)\bigr)$ and show 
it commutes with $F_{\mathcal{B}}(d')={d'}^\lambda$.  
By \eqref{BB3}, 
$F_{\mathcal{B}}\bigl((t_2t_1t_3t_2)^{-1}d(t_2t_1t_3t_2)\bigr)$ 
is equal to
\begin{equation}\label{BB4}
f(t_3^2,w_6)(t_2t_1t_3t_2)^\lambda f(w_6,t_3^2)\,d^\lambda\,f(t_3^2,w_6)(t_2t_1t_3t_2)^\lambda f(w_6,t_3^2).
\end{equation}
Now, both $d$ and $d'$ commute with $t_3=a_6a_5a_7a_6$ since they commute with $a_j$ for all $j\ne 4$. 
Furthermore, $d$ and $d'$ also commute with $w_6$. To see
this, we write $w_6=y_6y_5y_4y_3y_2$. By Lemma \ref{useful0}, we have
\begin{equation*}
\begin{cases}
w_6\,d\,w_6^{-1}=y_6y_5\,d\,y_5^{-1}y_6^{-1}=y_6\,d'\,y_6^{-1}=d\\
w_6\,d'\,w_6^{-1}=y_6y_5\,d'\,y_5^{-1}y_6^{-1}=y_6\,d\,y_6^{-1}=d'.
\end{cases}
\end{equation*}
This shows that both $d$ and $d'$ commute with $f(t_3^2,w_6)$. The commutation 
of this term with $d$ simplifies \eqref{BB4} to
\begin{equation}\label{FrhsCprime}
F_{\mathcal{B}}\bigl((t_2t_1t_3t_2)^{-1}d(t_2t_1t_3t_2)\bigr)=
f(t_3^2,w_6)(t_2t_1t_3t_2)^{-1}d(t_2t_1t_3t_2)f(w_6,t_3^2).
\end{equation}
To show that $F$ respects relation (C$'$), we must show that $F(d')={d'}^\lambda$ commutes with the
right-hand side of \eqref{FrhsCprime}. But in fact $d'$ already commutes with the right-hand
side of \eqref{FrhsCprime}; indeed we just saw that it commutes with the outer term
$f(w_6,t_3^2)$, and it commutes with the inner term
$(t_2t_1t_3t_2)d(t_2t_1t_3t_2)^{-1}$ since this is precisely what relation (C$'$) says. 
This concludes the proof that $F$ respects relation (C$'$), and thus completes the 
proof of the main Theorem \ref{thm11}.
\hfill{$\square$}

\section{The case of $\GGzero$}

Recall that we have a natural 
exact sequence
$$
1\to \hat\pi_{g,0} \to \GGone \to \GGzero\to 1
$$
where $\hat\pi_{g,0}$ is the profinite completion of 
the fundamental group of a closed surface of 
genus $g$. 
The goal of this section is to prove the following 

\begin{thm} 
\label{GGzero}
The $\GT$-action on $\GGone$ in
Theorem \ref{thm11} induces an automorphism group of 
$\GGzero$.
\end{thm}

To prove this, we must show that the action of $\GT$ on $\GGone$ 
passes to the group $\GGzero$ which according to Wajnryb \cite{W} 
is obtained from $\GGone$ by 
quotienting by a single further relation, presented in the following theorem.

\begin{thm}[Wajnryb \cite{W}] 
\label{wajnryb-ii}
To give a presentation of $\widehat\Gamma_{g,0}$, it suffices to add a single
relation (D) to the presentation of $\GGone$ in Theorem \ref{wajnryb} as follows. 
Let 
\begin{align*}t_i&=a_{2i}a_{2i-1}a_{2i+1}a_{2i}\hbox{ for }i=1,\ldots,g-1\\
v_1&=d'\hbox{ and }v_i=(t_{i-1}t_i)v_{i-1}(t_{i-1}t_i)^{-1}\hbox{ for }i=2,\ldots,g-1\\
u_i&=a_{2i}a_{2i+1}a_{2i+2}v_i(a_{2i+2}a_{2i+1}a_{2i}a_{2i-1})^{-1}\hbox{ for }i=1,\ldots,g-1,
\end{align*}
where $d'=d_2'$ is defined in \eqref{wajnrybdefs} of Theorem 
\ref{wajnryb}.
Let 
$$d_g=(u_1\cdots u_{g-1})^{-1}\cdot a_1\cdot (u_1\cdots u_{g-1})$$
and
$$y_g=a_{2g}a_{2g-1}\cdots a_2a_1^2a_2\cdots a_{2g-1}a_{2g}.$$
Finally, let $d'_g=y_{2g+1}d_gy_{2g+1}^{-1}$. Then

\vspace{.1cm}\noindent
(D) \ \ \ \ $d'_g=d_g$.
\end{thm}

\vspace{.1cm}
The relation (D) can be understood pictorially as what happens when there
is only a single hole on the right in Figure~2, so that the surface is of
type $(g,1)$, and we cap that hole to create a surface of type $(g,0)$. 
It is obvious that the two loops $\delta_g$ and $\delta'_g$ are identified
in that case, and Wajnryb's theorem tells us that no further relations are
needed in the presentation of $\GGzero$.

\vspace{.1cm}
As explained in \cite[Sect. 4]{LNS}, to prove the theorem, it suffices to show
the following.

\begin{prop}[Proposition 4.1 of \cite{LNS}]
\label{F(di)=di}
Let $F\in\GT$, and for $i=2,\ldots,g$ let $d_i$ denote the Dehn twist along the loop $\delta_i$ shown in
Figure 2. Then the automorphism $F$ of $\GGone$ given in \eqref{Faction} satisfies $F(d_i)=d_i^\lambda$.
\end{prop}

\noindent
In \cite{LNS}, the proof of this Proposition was omitted for reasons of
space. We give the complete proof below, but first we show why this
result is enough to prove Theorem \ref{GGzero}. 

\begin{proof}[Proof of Theorem \ref{GGzero} by using Proposition \ref{F(di)=di}]
Wajnryb's Theorem \ref{wajnryb-ii} above shows that the only relation that 
must be added to the presentation of $\GGone$ to obtain a presentation of 
$\GGzero$ is the relation $d_g=d'_g$. Since $d'_g=y_{2g+1}^{-1}d_gy_{2g+1}$, 
this relation is equivalent to the commutation relation $d_gy_{2g+1}=
y_{2g+1}d_g$ in $\GGzero$. But since $F(d_g)=d_g^\lambda$ by Proposition 
\ref{F(di)=di}, and $F(y_{2g+1})=y_{2g+1})^\lambda$ 
by Proposition \ref{braidrels}, the images $F(d_g)$ and $F(y_{2g+1})$ so
indeed commute, completing the proof of Theorem \ref{GGzero}.
\end{proof}

The rest of this section will be devoted to proving
Proposition \ref{F(di)=di}. 
The fact that $F(d_2)=d_2^\lambda$ comes from \eqref{Faction}, and we proved that
$F(d_3)=d_3^\lambda$ in \S 5.2. 
We will use the same method to prove that $F(d_i)=d_i^\lambda$ by induction.
The method is as follows. Given $i\ge 3$,
we cut out a 
subsurface $S_i$ of genus $0$ with four boundary components along the loops $\delta_{i-1}$, $\alpha_{2i-1}$,
$\alpha_{2i+1}$ and $\delta_{i+1}$. The loop $\delta_i$ sits on this subsurface, dividing the boundary
components $\delta_{i-1}$ and $\alpha_{2i-1}$ from the boundary components $\alpha_{2i+1}$ and $\delta_{i+1}$.
We saw in Figure 3 that the (inverse of the) braid 
$t_i=a_{2i}a_{2i-1}a_{2i+1}a_{2i}$ acts on 
the loop $\delta_i$ drawn in the
top diagram of Figure 3 by deforming it to the loop 
$\gamma_{1,i}$
in the middle diagram, and the braid $t'_{i-1}
=a_{2i-2}d_{i-1}a_{2i-1}a_{2i-2}$ then deforms that loop 
to the loop $\gamma_{2,i}$ in the lowest diagram.
The corresponding twists can be written as
\begin{equation*}
g_{1,i}:=t_i^{-1}d_it_i, \qquad
g_{2,i}:=(t_2\cdots t_it_1\cdots t_{i-1} )^{-1}d_2(t_2\cdots t_it_1\cdots t_{i-1}).
\end{equation*}
In fact, the above expression of $g_{2,i}$ is given in 
Wajnryb's paper \cite[p.166 line\,$-2$]{W}, where
our loop $\gamma_{2,i}$ corresponds to
a loop denoted $\delta_{i,i+1}$ in \cite{W}.
Later we will use an inductive identity 
$g_{2,i+1}=t_{i}^{-1} t_{i+1}^{-1}g_{2,i}t_{i+1} t_{i}$
that immediately follows from this expression
(cf. \cite[p.165 (10)]{W}).

Since the three loops $\gamma_{2,i}$, $\gamma_{1,i}$ and $\delta_i$ 
form a lantern in the subsurface cut out by loops $\delta_{i-1}$,
$\alpha_{2i-1}$, $\alpha_{2i+1}$ and $\delta_{i+1}$, 
we have the lantern relation 
\begin{equation}
\label{lantern_i}
g_{2,i}g_{1,i} d_i=d_{i-1}a_{2i-1}a_{2i+1}d_{i+1},
\end{equation}
and since we know the action of $F$ on all the $a_i$ and on $d_2$ (and
$d_3$), we can make the inductive hypothesis that $F(d_i)=d_i^\lambda$ 
and use this equality to compute $F(d_{i+1})$.  

Let us define
$$\begin{cases}
g'_{1,i}:=t_i^{-1}d'_it_i\cr
g'_{2,i}:=(t_2\cdots t_it_1\cdots t_{i-1})^{-1}d'_2(t_2\cdots t_it_1\cdots t_{i-1}).
\end{cases}$$
Then we have the following result.

\begin{lem}
\label{GTaction_g1-g2}
Let $F=(\lambda,f)\in\GT$ and consider the action
of 
$\FB=\mathrm{inn}(\prod_{j=1}^{g-1} f(y_{2j+1},a_{2j+1}^2))\circ F$
on $\GGone$ given in Theorem \ref{thm11}.
Write $m:=\frac{\lambda-1}{2}$.
Then, for $3\le i\le g-1$, we have
\begin{enumerate}
\item[(i)] 
$\FB(g_{1,i})=(d_id_i')^m
f(g_{1,i}g_{1,i}',d_id_i')g_{1,i}^\lambda 
f(d_id_i',g_{1,i}g_{1,i}')(d_id_i')^{-m};$
\item[(ii)]
$\FB(g_{2,i})=f(g_{2,i}g_{2,i}',d_id_i')g_{2,i}^\lambda 
f(d_id_i',g_{2,i}g_{2,i}')$.
\end{enumerate}
\end{lem}

\noindent
Note that the right hand sides of
Lemma \ref{GTaction_g1-g2} can be written as
\begin{equation}
\label{HaikuB}
\FB(g_{1,i})=(d_i)^m
f(g_{1,i},d_i)g_{1,i}^\lambda 
f(d_i,g_{1,i})(d_i)^{-m}; \qquad
\FB(g_{2,i})=f(g_{2,i},d_i)g_{2,i}^\lambda 
f(d_i,g_{2,i})
\end{equation}
by using (iii) of the Haiku Lemma.

\vspace{.3cm}
This lemma, along with the induction hypothesis $\FB(d_k)=d_k^\lambda$
for $k\le i$, will play here exactly the same role 
in the calculation of $F(d_{i+1})$ as \eqref{g1g2} in subsection \ref{CompFd3}
played to pass from $d_2$ to $d_3$. Before giving its proof, we show how
to use it to complete the proof of Proposition \ref{F(di)=di}.

\begin{proof}[Proof of Proposition \ref{F(di)=di}
by using Lemma \ref{GTaction_g1-g2}]

The formula \eqref{HaikuB} allows us to calculate
the action of $\FB$ on the left hand side 
$\Delta:=g_{2,i}g_{1,i}d_i$ of the lantern relation \eqref{lantern_i} 
as
\begin{align*}
\FB(\Delta)=\FB(g_{2,i}g_{1,i}d_i)
&=
f(g_{2,i},d_i)g_{2,i}^\lambda f(d_i,g_{2,i})
d_i^m f(g_{1,i},d_i)g_{1,i}^\lambda f(d_i,g_{1,i})d_i^{-m}
d_i^\lambda \\
&=
f(g_{2,i},d_i)g_{2,i}^\lambda 
\cdot
g_{2,i}^{-m}f(g_{1,i}, g_{2,i}) g_{1,i}^{-m}\Delta^m 
\cdot
g_{1,i}^\lambda
f(d_i,g_{1,i})d_i^{-m}
d_i^\lambda 
\\
&=
f(g_{2,i},d_i)g_{2,i}^{\frac{\lambda+1}{2}}
f(g_{1,i}, g_{2,i}) g_{1,i}^{\frac{\lambda+1}{2}}
 f(d_i,g_{1,i})
d_i^{\frac{\lambda+1}{2}}
\Delta^m 
\\
&=
\Delta^{\frac{\lambda+1}{2}}\Delta^m =\Delta^\lambda,
\end{align*} 
where we used relation (II) 
(in the form of \cite{NS} Lemma 1.5 (1.5.2)) for the second equality.
By the inductive hypothesis, the action on the right hand side 
of \eqref{lantern_i} is $d_{i-1}^\lambda a_{2i-1}^\lambda
a_{2i+1}^\lambda \FB(d_{i+1})$; hence we get $\FB(d_{i+1})=d_{i+1}^\lambda$.
This settles the proof of Proposition
\ref{F(di)=di}. 
\end{proof}

To complete the proof of Theorem \ref{GGzero}, it remains 
to prove Lemma \ref{GTaction_g1-g2}.
Recall that the elements $t_1=a_2a_1a_3a_2$,
$t_2=a_4a_3a_5a_4$, $\dots$ 
induced from the original braids $a_1,a_2,\dots$
behave like ribbon braids, flat crossing the underlying strands
in pairs.
In particular, they satisfy the braid relations
$t_it_{i+1}t_i=t_{i+1}t_it_{i+1}$,
$t_it_j=t_jt_i$ ($|i-j|>1$).

Set
\begin{align*}
&Y_k:=t_{k-1}\cdots t_1\, t_1\cdots t_{k-1};\\
&W_k:=(t_1t_2\cdots t_{k-1})^k=Y_2 Y_3 \cdots Y_k
\end{align*}
for $k\ge 2$.

\begin{prop} \label{AppBprop}
Notations being as above,
the inside factors of Lemma \ref{GTaction_g1-g2}
can be written as follows.
\begin{enumerate}
\item
$d_id_i' =w_{2i}=W_ia_1^2a_3^2\cdots a_{2i-1}^2$.

\item $g_{1,i}g_{1,i}'=t_i^{-1}w_{2i}t_i$.

\item
$g_{2,i}g_{2,i}'=t_i^2a_{2i-1}^2a_{2i+1}^2$.

\end{enumerate}
\end{prop}

For the proof of this proposition, 
the following lemma will be useful.
\begin{lem}
\label{exLemma}
\begin{enumerate}
\item[(i)] If $i\ne j$ then $W_j$ commutes with $t_i$.
\item[(ii)] If $i\ne j$ then $d_j$, $w_{2j}$ 
commute with both 
$g_{1,i}$ and $g_{2,i}$.
\item[(iii)] $(t_iW_i)^2=W_{i+1}W_{i-1}$.
\item[(iv)] $t_iy_{2i}y_{2i-1}t_i=y_{2i+2}y_{2i+1}=Y_{i+1}a_{2i+1}^2$.
\end{enumerate}
\end{lem}

\begin{proof}
(i) By definition $W_j=(t_1\cdots t_{j-1})^j$, and this element commutes
with $t_i$ for all $j\ne i$ inside the ribbon-braid group generated by
$t_1,t_2,\dots$. 
For (ii), recall that $d_i$, $g_{1,i}$ and $g_{2,i}$
are twists along the loops $\delta_i$, $\gamma_{1,i}$ and 
$\gamma_{2,i}$ of Figure 3 which lie on the subsurface
$S_i$ of genus $0$ with four boundary components along the loops $\delta_{i-1}$, $\alpha_{2i-1}$,
$\alpha_{2i+1}$ and $\delta_{i+1}$. 
The assertion is clear from the locations of $S_i$
and $S_j$ when $i\ne j$ and from the fact
$w_{2j}=d_jd_j'$.
For (iii), we have 
$$(t_iW_i)^2=t_iY_iY_{i-1}\cdots Y_2t_iW_i
=(t_iY_it_i)Y_{i-1}\cdots Y_2(Y_iW_{i-1})
=Y_{i+1}Y_iY_{i-1}\cdots Y_2W_{i-1}
=W_{i+1}W_{i-1},$$
which implies the assertion.
Finally, the first equality of (iv) follows from the following computation
using only braid commutations and the usual braid relations (indicated
by the underlined portions).
\begin{align*}
t_iy_{2i}y_{2i-1}t_i &= 
a_{2i}a_{2i+1}\underline{a_{2i-1}
a_{2i}\cdot
a_{2i-1}}y_{2i-1}a_{2i-1}\cdot
y_{2i-1}\cdot
a_{2i}a_{2i+1}a_{2i-1}a_{2i}
\\
&=\underline{a_{2i}a_{2i+1}a_{2i}}
a_{2i-1}a_{2i}\,
y_{2i-1}a_{2i-1}\cdot
a_{2i}a_{2i+1}y_{2i-1}a_{2i-1}a_{2i}
\\
&=a_{2i+1}a_{2i}a_{2i+1}
a_{2i-1}a_{2i}\,
y_{2i-1}
a_{2i-1}\cdot
a_{2i}a_{2i+1}y_{2i-1}a_{2i-1}a_{2i}
\\
&=a_{2i+1}a_{2i}a_{2i-1}y_{2i-1}
\Bigl(a_{2i+1}\underline{a_{2i} a_{2i-1}\cdot a_{2i}}a_{2i+1}\Bigr)
y_{2i-1}a_{2i-1}a_{2i}\\
&=a_{2i+1}a_{2i}a_{2i-1}y_{2i-1}
\Bigl(a_{2i+1}a_{2i-1} a_{2i}a_{2i-1}a_{2i+1}\Bigr)
y_{2i-1}a_{2i-1}a_{2i}\\
&=a_{2i+1}a_{2i}a_{2i-1}y_{2i-1}
\Bigl(a_{2i-1}\underline{a_{2i+1} a_{2i}a_{2i+1}}a_{2i-1}\Bigr)
y_{2i-1}a_{2i-1}a_{2i}\\
&=a_{2i+1}a_{2i}a_{2i-1}y_{2i-1}\Bigl(a_{2i-1}a_{2i}a_{2i+1}a_{2i}
a_{2i-1}\Bigr)
y_{2i-1}a_{2i-1}a_{2i}\\
&=y_{2i+2}y_{2i+1}.
\end{align*}
For the second equality of (iv), we use induction on $i$. When $i=1$, 
it is easy to see that $y_4y_3=t_1^2a_3^2=Y_2^2a_3^2$. 
Suppose the desired equality for $i=k$.
Then since $t_k(a_{2k\pm 1})t_k^{-1}=a_{2k\mp 1}$, it follows that 
$t_ky_{2k}y_{2k-1}t_k=t_kY_{k}a_{2k-1}^2t_k=t_kY_{k}t_ka_{2k+1}^2
=Y_{k+1}a_{2k+1}^2$.
This completes the proof of the Lemma.
\end{proof}

\begin{proof}[Proof of Proposition \ref{AppBprop}]
(1) It suffices to prove the second equality 
$w_{2i}=W_i a_1^2\cdots a_{2i-1}^2$ by induction on $i$.
When $i=2$, $w_4=t_1^2a_1^2a_3^2=W_2a_1^2a_3^2$ holds.
Suppose it is true for $i\le k$. 
By Lemma \ref{exLemma} (iv) it follows that 
$y_{2k+2}y_{2k_+1}=t_ky_{2k}y_{2k-1}t_k=
Y_{k+1}a_{2k+1}^2$. Then
\begin{align*}
w_{2k+2}&=y_{2k+2}y_{2k_+1}\cdot w_{2k} \\
&=Y_{k+1}a_{2k+1}^2\cdot W_{k}a_1^2\cdots a_{2k-1}^2\\
&=Y_{k+1}W_k a_1^2\cdots a_{2k+1}^2=W_{k+1} a_1^2\cdots a_{2k+1}^2
\end{align*}
(2) follows from 
$g_{1,i}g_{1,i}'=t_i^{-1}(d_id_i')t_i$ and 
$d_id_i'=w_{2i}$. 
(3) For $i=2$, we have $g_{2,2}g_{2,2}'=w_4=t_1^2a_1^2a_3^2$.
Assume it holds for $i=k$, then,
\begin{align*}
g_{2,k+1}g_{2,k+1}'&=t_k^{-1}t_{k+1}^{-1}(g_{2,k}g_{2,k}')t_{k+1}t_k \\
&=t_k^{-1}t_{k+1}^{-1}(t_k^2 a_{2k-1}^2 a_{2k+1}^2)t_{k+1}t_k \\
&=t_k^{-1}t_{k+1}^{-1}t_k^2 a_{2k-1}^2 t_{k+1} a_{2k+3}^2 t_k\\
&=t_k^{-1}t_{k+1}^{-1}t_k^2 a_{2k-1}^2 t_{k+1} t_k a_{2k+3}^2\\
&=t_k^{-1}t_{k+1}^{-1}t_k^2  t_{k+1} t_k a_{2k+1}^2a_{2k+3}^2 =
t_{k+1}^2a_{2k+1}^2a_{2k+3}^2 ,
\end{align*}
hence the case of $i=k+1$ also follows.
\end{proof}

By virtue of Proposition \ref{AppBprop}, 
one finds that the formulas of Lemma \ref{GTaction_g1-g2}
are equivalent to
$$ 
\begin{cases}
\FB(g_{1,i})&=(W_i)^m
f(t_i^{-1}W_it_i,W_i) g_{1,i}^\lambda 
f(W_i,t_i^{-1}W_it_i) W_i^{-m}; \\
\FB(g_{2,i})&=f(t_i^2,W_i)g_{2,i}^\lambda 
f(W_i,t_i^2).
\end{cases}
$$

\begin{figure}[h!]
    \centering
    \includegraphics[width=0.8\textwidth]{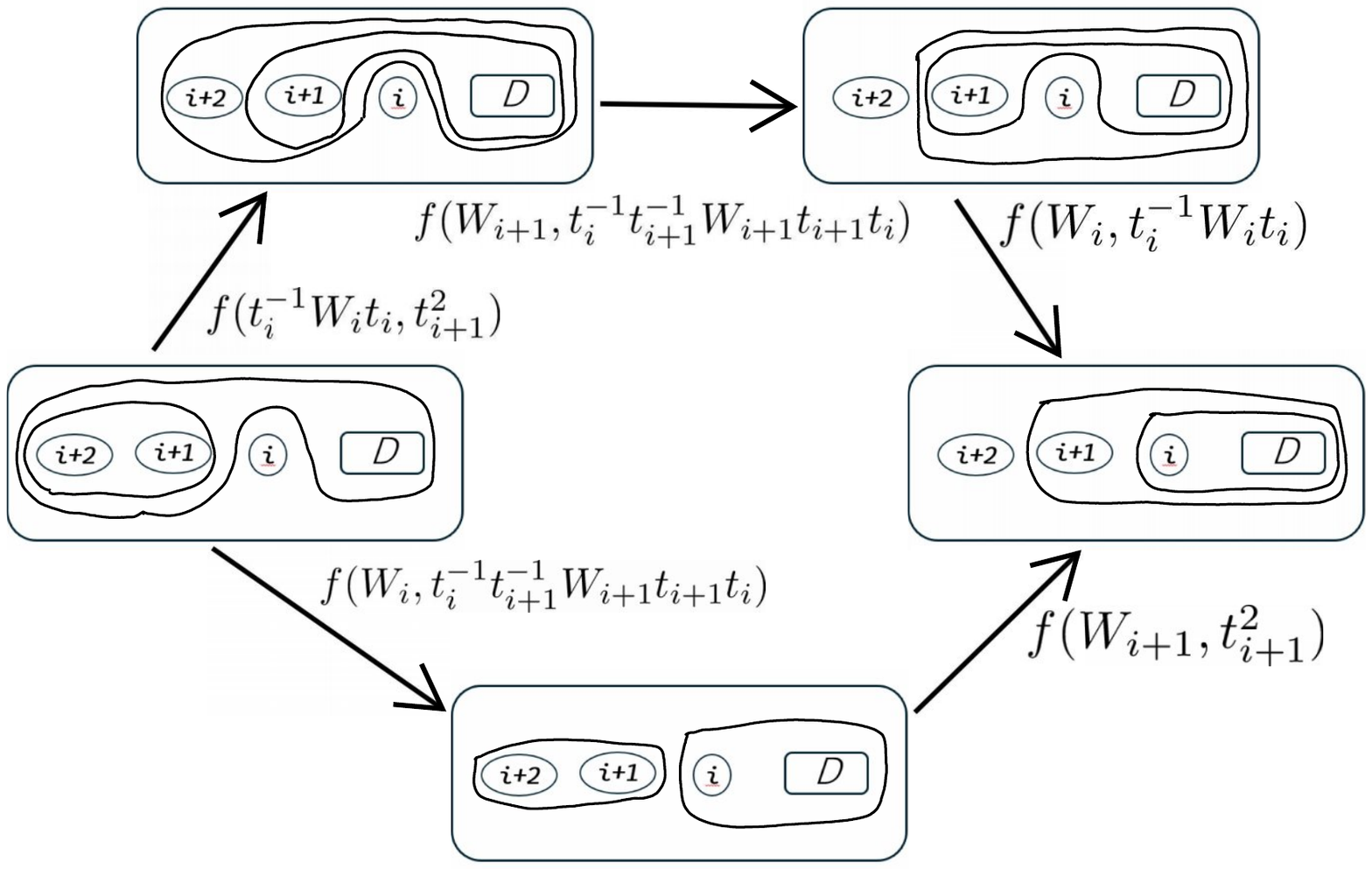}
    \caption{A-moves at the level of ribbon braids}
    \label{fig:diamond}
\end{figure}

\begin{proof}[Proof of Lemma \ref{GTaction_g1-g2}]

Below we use the notation $[x]\!\ast y:=y^{-1}xy$ for simplicity.
The assertion (i) follows from the computation for $i\ge 2$:
\begin{align*}
\FB(g_{1,i})&=\FB(t_i^{-1}d_it_i) \\
&=[d_i^\lambda]\!\ast f(t_i^2,W_i)t_i^\lambda f(W_i,t_i^2) \\
&=[d_i^\lambda]\!\ast t_i f(t_i^2,t_i^{-1}W_it_i) t_i^{\lambda-1}
f(W_i,t_i^2) \\
&=[d_i^\lambda]\!\ast t_i (t_iW_i)^2
f(W_i,t_i^{-1}W_it_i) W_i^{-m}\\
&=[t_i^{-1}d_i^\lambda t_i]\!\ast
f(W_i,t_i^{-1}W_it_i) W_i^{-m},
\end{align*}
where the fact that $d_i,t_i$ commute with 
$(t_iW_i)^2=W_{i-1}W_{i+1}$ (Lemma \ref{exLemma})
is used in the last equality.
We prove (ii) by induction on $i$. 
For $i=2$, it is already shown in \eqref{g1g2}.
Suppose $\FB(g_{2,k})$ ($k\le i$) are given in the
form of assertion. 
Then, 
\begin{align*}
\FB(g_{2,i+1})=&\FB(t_{i}^{-1} t_{i+1}^{-1}g_{2,i}t_{i+1} t_{i}) 
=\bigl[ \FB(g_{2,i}) \bigr]\!\ast \FB(t_{i+1}) \FB(t_i) \\
=&[g_{2,i}^\lambda]\!\ast f(W_i,t_i^2)\cdot
f(t_{i+1}^2,W_{i+1}) t_{i+1}^\lambda f(W_{i+1},t_{i+1}^2)
\cdot
f(t_{i}^2,W_{i}) t_{i}^\lambda f(W_{i},t_{i}^2)\\
=&
[\underline{g_{2,i}^\lambda}]\!\ast\Bigl( 
\uwave{f(W_i,t_i^2)}\cdot 
\underline{t_{i+1}
(t_{i+1}^{-1}W_{i+1}^{-m}} 
\uwave{t_{i+1}}) 
f(W_{i+1},t_{i+1}^{-1}W_{i+1}t_{i+1})
W_{i+1}^{-m}
(t_{i+1}W_{i+1})^{2m} \\
&\cdot
t_{i}(t_{i}W_{i})^{2m}(t_{i}^{-1}W_{i}^{-m}t_{i}) 
f(W_{i},t_{i}^{-1}W_{i}t_{i})W_{i}^{-m}\Bigr) 
\end{align*}
Since the underlined portion should vanish as
$W_{i+1}$ commutes with $g_{2,i}, W_i$ and $t_i^2$
by Lemma \ref{exLemma}, and since 
the underwaved portion can then be rewritten as
$$
t_{i+1}f(W_i,t_{i+1}^{-1}t_i^2t_{i+1})=t_{i+1}f(W_i,t_it_{i+1}^2t_i^{-1})
=t_{i+1}t_i f(t_i^{-1}W_it_i,t_{i+1}^2)t_i^{-1},
$$
the above calculation continues to
\allowdisplaybreaks
\begin{align*}
\FB(g_{2,i+1})=&
[g_{2,i}^\lambda]\!\ast\Bigl(
t_{i+1}t_i
f(t_i^{-1}W_it_i,t_{i+1}^2)
t_i^{-1}f(W_{i+1},t_{i+1}^{-1}W_{i+1}t_{i+1})
W_{i+1}^{-m}
(W_iW_{i+2})^m \\
&\cdot
t_{i}(W_{i-1}W_{i+1})^m(t_{i}^{-1}W_{i}^{-m}t_{i}) 
f(W_{i},t_{i}^{-1}W_{i}t_{i})W_{i}^{-m}\Bigr) \\
=&
[g_{2,i}^\lambda]\!\ast\Bigl(
t_{i+1}t_i
f(t_i^{-1}W_it_i,t_{i+1}^2)
f(W_{i+1},t_i^{-1}t_{i+1}^{-1}W_{i+1}t_{i+1}t_i)
W_{i+1}^{-m}
t_i^{-1}(W_iW_{i+2})^m t_i \\
&\cdot
(W_{i-1}W_{i+1})^m(W_{i}^{-m}) (t_{i}^{-1}W_{i}^{-m}t_{i}) 
f(t_iW_{i}t_{i}^{-1},W_{i})W_{i}^{-m}\Bigr) 
\\
=&
[g_{2,i+1}^\lambda]\!\ast\Bigl(
f(t_i^{-1}W_it_i,t_{i+1}^2)
f(W_{i+1},t_i^{-1}t_{i+1}^{-1}W_{i+1}t_{i+1}t_i)
f(t_iW_{i}t_{i}^{-1},W_{i})W_{i}^{-m}\Bigr) \\
=&
[g_{2,i+1}^\lambda]\!\ast\Bigl(
f(W_i,t_i^{-1}t_{i+1}^{-1}W_{i+1}t_{i+1}t_i)
f(W_{i+1},t_{i+1}^2)
W_{i}^{-m}\Bigr)\\
=&
[g_{2,i+1}^\lambda]\!\ast
f(W_{i+1},t_{i+1}^2)
=
f(t_{i+1}^2,W_i)
g_{2,i+1}^\lambda
f(W_{i+1},t_{i+1}^2)
.
\end{align*}
Note that relation (III) along the diagram of
Figure \ref{fig:diamond} is used for the above
second to last line, and commutativity of 
$g_{2,i+1}=t_i^{-1}t_{i+1}^{-1}g_{2,i}t_{i+1}t_i$ with
$W_i$ and $t_i^{-1}t_{i+1}^{-1}W_{i+1}t_{i+1}t_i$ 
(by Proposition \ref{AppBprop} (1) 
and Lemma \ref{exLemma} (ii))
for the last line.
The proof of assertion (ii) is completed. 
\end{proof}

\vspace{.2cm}
\section{Prospects for the general case of $\GGn$, $n>1$}

Let $\Sigma_{g,n}$ denote an oriented surface of genus $g$ and $n$ boundary components,
and consider the profinite mapping class group $\widehat\Gamma_{g,n}$ of it.
For $F=(\lambda, f)\in\GT$, knowing that $F(d_i)=d_i^\lambda$ for 
$i=1,\ldots,g$ in our standard action on $\GGone$ (Theorem \ref{thm11} and
Proposition \ref{F(di)=di}),
we can then cut along these loops in several ways to obtain subsurfaces
$\Sigma_{g',n'}$ of $\Sigma_{g,1}$
of smaller genus $g'$ with more than one boundary component. 
If the action of $F\in \GT$ on $\GGone$ restricts to {\it the image of}
$\widehat\Gamma_{g',n'}$, and if we take as a working hypothesis that the map
from $\widehat\Gamma_{g',n'}$ to its image inside $\GGone$ is injective, then
(ignoring small differences of effects of
punctures vs boundary components on mapping class groups of surfaces),
this would show that $F$ acts on $\widehat\Gamma_{g',n'}$, generalizing our main
result to the case of all $(g,n)$.

\begin{figure}[h!]
    \centering
    \includegraphics[width=0.8\textwidth]{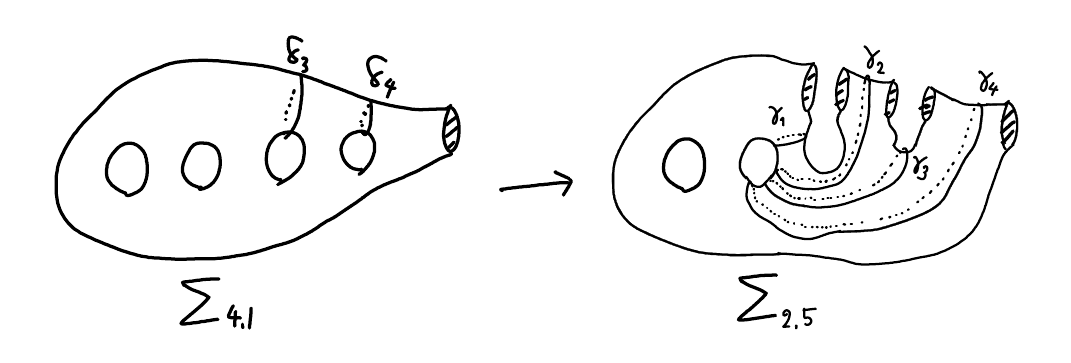}
    \caption{Surfaces  $\Sigma_{2,5}$ cut off from $\Sigma_{g,n}=\Sigma_{4,1}$}
    \label{fig:prospect}
\end{figure}

\vspace{.2cm}

A typical example of this cutting along loops to create more boundary components is shown in
Figure \ref{fig:prospect}, where the loops $\delta_3$ and $\delta_4$ on $\Sigma_{4,1}$ have been
cut, yielding a subsurface $\Sigma_{2,5}$.
The mapping class group $\Gamma_{2,5}$ is generated by the Gervais twists
along loops indicated in Figure \ref{fig:Fish}, and the $\GT$-action on the 
twists along the loops $\gamma_2,\gamma_3,\gamma_4$ shown in Figure \ref{fig:prospect}
should, in principle, be identical to the action on the same twists along the same loops
viewed as sitting on $\widehat\Gamma_{4,1}$.  If one shows that $F(\gamma_i)$ (for $i=2,3,4$) 
can be expressed in $\widehat\Gamma_{2,5}$ for every $F=(\lambda,f)\in\GT$ acting on 
$\widehat\Gamma_{4,1}$ as in Theorem \ref{thm11}, then we obtain a $\GT$-action on the image of 
$\widehat\Gamma_{2,5}\to \widehat\Gamma_{4,1}$.  At the writing of this paper, we do not know if this 
holds for $\GT$. However, we do know that it holds for a particular subgroup $\mathrm{I\!\Gamma}$ of
$\GT$ introduced in \cite{HLS}, \cite{LNS}, \cite{NS}, where Grothendieck's ``game of Lego'' is explained,
and it is shown that the $\mathrm{I\!\Gamma}$-action is a ``Lego action'' in the sense that it respects
the rules of this game (\cite{Gr}).  One important feature of the Lego game is that the 
$\mathrm{I\!\Gamma}$-action on $\widehat\Gamma_{g,n}$ for all $(g,n)$ respects all the homomorphisms 
between the profinite mapping class groups induced by cutting along simple loops, like the example
$\widehat\Gamma_{2,5}\rightarrow \widehat\Gamma_{4,1}$ above. The question of whether the
``Lego action'' of $\mathrm{I\!\Gamma}$ extends to a Lego action of the full group $\GT$ acting on
$\GGone$ as in Theorem \ref{thm11} is the subject of a future research project of the authors of 
this paper.

\appendix
\section{Direct derivation of \eqref{t1t2t3}}

In the above text, we used the following action of 
$\GT$ in \eqref{t1t2t3} as a consequence of the
more general ``A-move formula''  
(Theorem \ref{Amovetheorem}):

\begin{equation}\label{Appt1t2t3}
\begin{cases}
F_{\mathcal{B}}(t_1)=t_1^\lambda,
\\
F_{\mathcal{B}}(t_2)=f(t_2^2,t_1^2)\,t_2^\lambda \,f(t_1^2,t_2^2),
\\
F_{\mathcal{B}}(t_3)=f(t_3^2,t_2t_1^2t_2)\,t_3^\lambda \,f(t_2t_1^2t_2,t_3^2),
\end{cases}
\end{equation}
where $t_i=a_{2i}a_{2i+1}a_{2i-1}a_{2i}$ ($i=1,2,3$).

In this appendix, we will show how these formula can
be derived from Drinfeld's $\GT$-action on the basic
braids $a_1,a_2,\dots$ directly.
Fix $n\ge 4$ and $F=(\lambda,f)\in\GT$ acting on the profinite
braid group 
$$
\widehat{B}_n=\left\langle 
a_1,\dots,a_{n-1}\left|
\begin{matrix} 
&a_ia_j=a_ja_i\quad (|i-j|>1), \\
&a_ia_{i+1}a_i=a_{i+1}a_ia_{i+1}\quad (i=1,\dots,n-2)
\end{matrix}
\right.\right\rangle.
$$
Below we shall write 
$$
\FB:=\left({\rm inn}\,\prod_{i=1}^{\lfloor (n-1)/2 \rfloor}
f(y_{2i+1},a_{2i+1}^2)
\right)\circ F
\quad\in  \mathrm{Aut}(\widehat{B}_n).
$$
We first show the simplest case:
\begin{prop}
Notations being as above,
$F_{\mathcal{B}}$ acts on $t_1=a_2a_1a_3a_2$ by
\begin{equation}
\FB(t_1)= t_1^\lambda.
\end{equation}

\end{prop}

\begin{proof}
We use the hexagon relation three times in the form
\cite[Lemma 1.5]{NS}:
\begin{equation}
\label{NSlem1.5}
f(x,y)x^m f(z,x) z^m f(y,z) y^m=\omega^m
\end{equation}
where each of $x,y,z$ is supposed commuting
with $\omega:=xyz$.

Now by definition $\FB$ acts on $a_1,a_2,a_3$ by 
\begin{align}
&\FB(a_1)=a_1^\lambda, \quad
\FB(a_3)=a_3^\lambda \\
&\FB(a_2)=f(w_3,a_3^2) f(a_2^2,a_1^2) a_2^\lambda
f(a_1^2,a_2^2)f(a_3^2,w_3) 
\end{align}
but we use two different expressions for $\FB(a_2)$:
\begin{align}
\FB(a_2)
&= f(w_3,a_3^2) 
\underline{f(a_2^2,a_1^2) a_2^{\lambda-1}
f(a_2a_1^2a_2^{-1},a_2^2)}
f(a_2a_3^2a_2^{-1},w_3)\cdot a_2
\label{form1}
\\
&=f(w_3,a_3^2) 
\underline{a_1^{1-\lambda}
f(a_2a_1^2a_2^{-1},a_1^2)
a_2a_1^{1-\lambda}a_2^{-1}
w_3^{m}}
f(a_2a_3^2a_2^{-1},w_3) \cdot a_2
\label{form2}
\end{align}
where \eqref{NSlem1.5} is applied to the underlined portion.

We compute $\FB(t_1)=\FB(a_2)\FB(a_1)\FB(a_3)\FB(a_2)$
along with the diagram in Figure \ref{fig:AmovesforT1},
where arrows between pants-decompositions
(of a punctured plane) have labels of $f(G,L)$ with
$G,L$ indicate twists of a gained loop and a lost loop
respectively.

\begin{figure}[h!]
    \centering
    \includegraphics[width=0.8\textwidth]{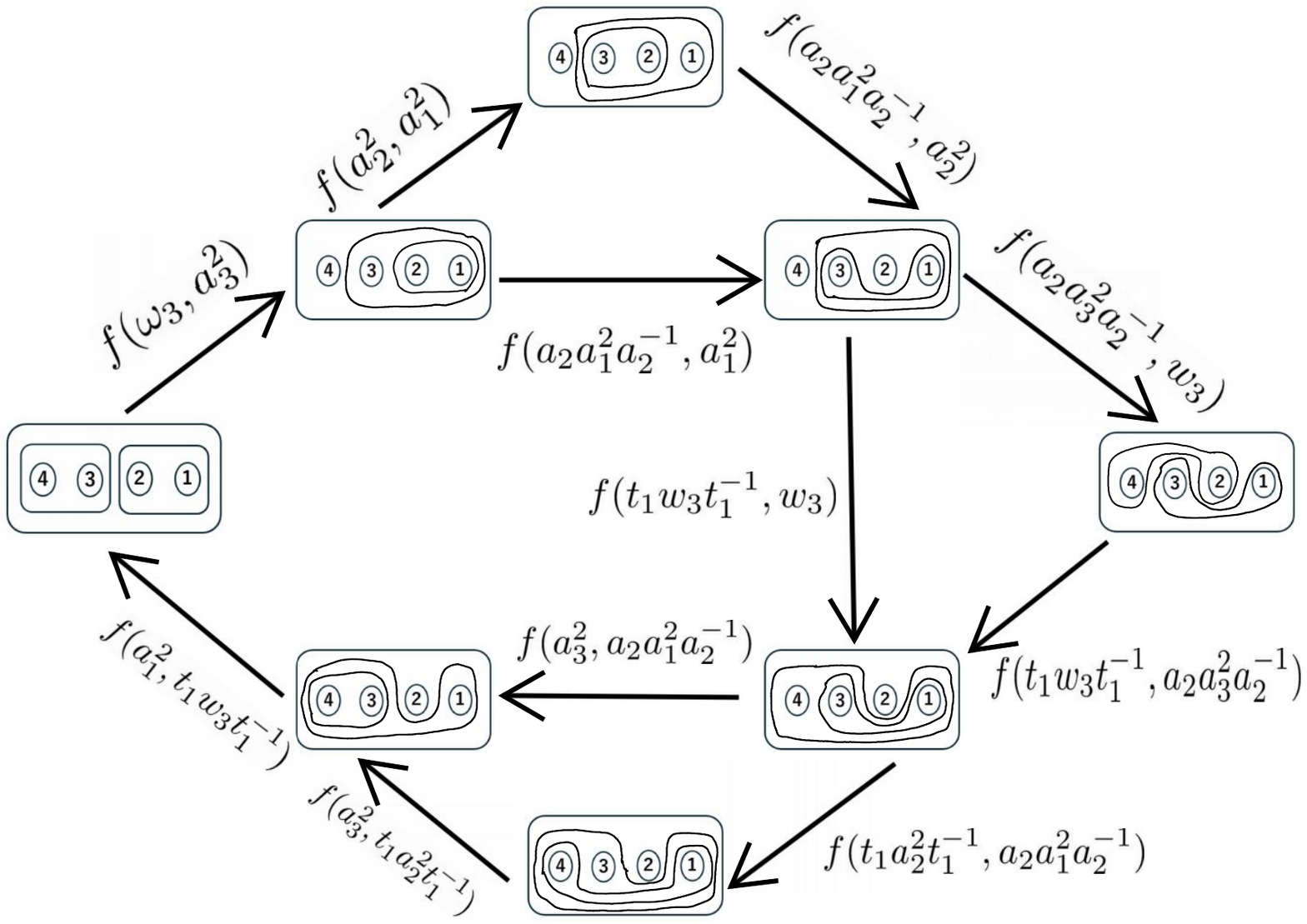}
    \caption{A-moves for $\FB(t_1)$}
    \label{fig:AmovesforT1}
\end{figure}

Let us apply \eqref{form2} to the first factor,
and \eqref{form1} to the last factor. 
We then get:
\begin{align}
&\FB(t_1)=\FB(a_2)\FB(a_1)\FB(a_3)\FB(a_2) 
\label{FBt1}
\\
=&
a_1^{1-\lambda}f(w_3,a_3^2) w_3^{m}
f(a_2a_1^2a_2^{-1},a_1^2)
a_2a_1^{1-\lambda}a_2^{-1}
f(a_2a_3^2a_2^{-1},w_3) 
\underline{
a_2a_1^{1-\lambda}a_2^{-1}\cdot a_2
\cdot (a_1^\lambda \cdot a_3^\lambda)}
\nonumber
\\
&\cdot f(w_3,a_3^2)
f(a_2^2,a_1^2) a_2^{\lambda-1}
f(a_2a_1^2a_2^{-1},a_2^2)
f(a_2a_3^2a_2^{-1},w_3)\cdot a_2
\nonumber
\end{align}

\begin{figure}[H]
    \centering
    \includegraphics[width=0.95\textwidth]{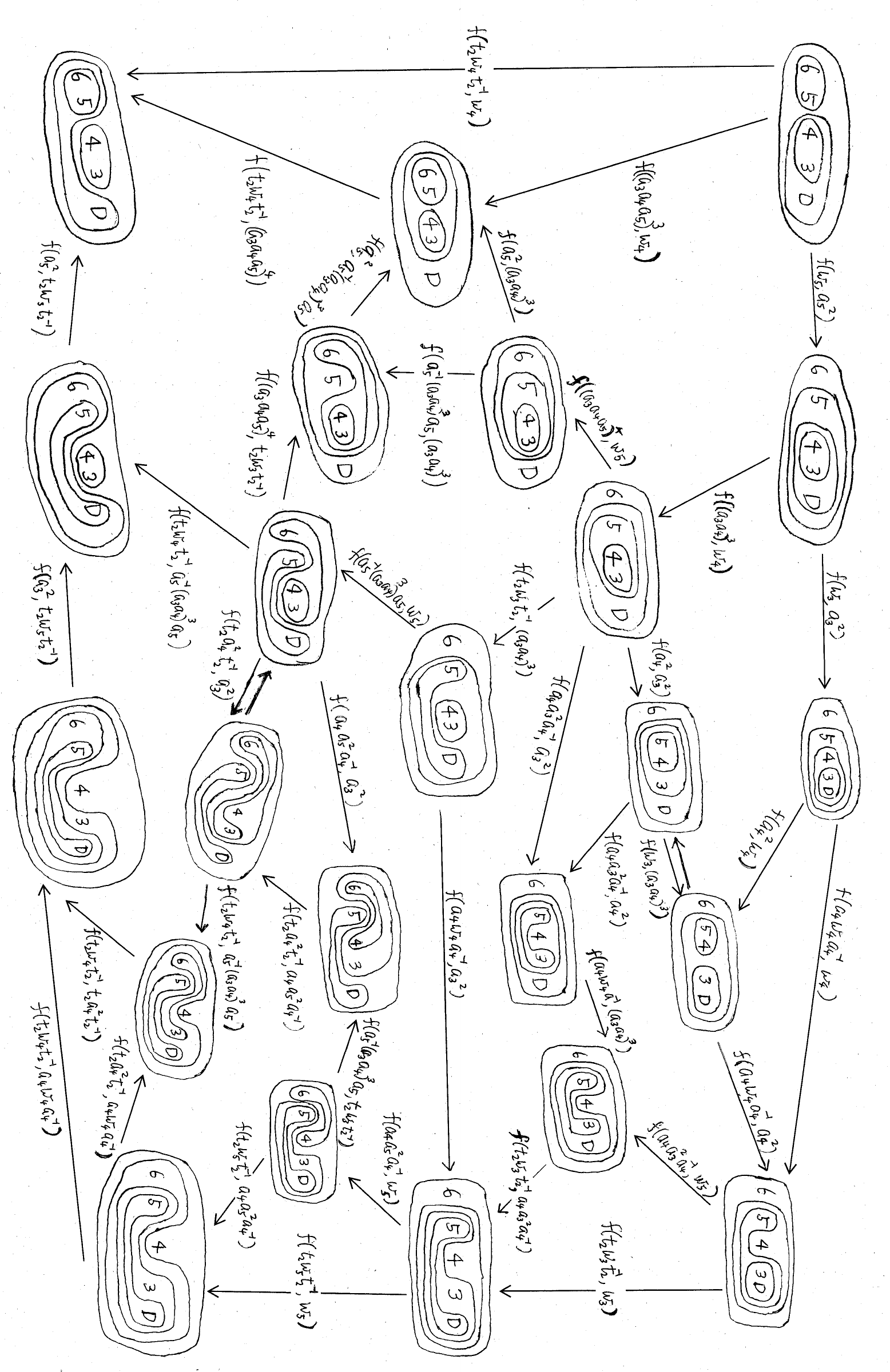}
  
    \caption{A-moves for $\FB(t_i)$ ($i\ge 2$)}
    \label{fig:AmovesforTi}
\end{figure}

Replacing the above underlined portion by
$a_2a_3^{\lambda-1}a_2^{-1}\cdot (a_2a_1a_3)$, it 
continues to 
\begin{align*}
=&
a_1^{1-\lambda}f(w_3,a_3^2) w_3^{m}
f(a_2a_1^2a_2^{-1},a_1^2) a_2a_1^{1-\lambda}a_2^{-1}
\uwave{f(a_2a_3^2a_2^{-1},w_3)\cdot 
a_2a_3^{\lambda-1}a_2^{-1}f(t_1w_3t_1^{-1},a_2a_3^2a_2^{-1})}
\\
&\cdot
\underline{f(t_1a_2^2t_1^{-1},a_2a_1^2a_2^{-1}) t_1a_2^{\lambda-1}t_1^{-1}
f(t_1a_1^2t_1^{-1},t_1a_2^2t_1^{-1})}
\cdot
f(t_1a_3^2t_1^{-1},t_1w_3t_1^{-1})\cdot t_1
\end{align*}
Noting that the conjugation by $t_1$ switches $a_1$ and $a_3$
and putting $\Omega:=(a_2a_3^2a_2^{-1})w_3(t_1w_3t_1^{-1})$, 
we further continue the above to
\begin{align*}
=&
a_1^{1-\lambda}f(w_3,a_3^2) w_3^{m}
f(a_2a_1^2a_2^{-1},a_1^2) 
\cdot
(a_2a_1^{1-\lambda}a_2^{-1})
\cdot
\uwave{
w_3^{-m}
f(t_1w_3 t_1^{-1},w_3)
(t_1w_3^{-m}t_1^{-1})\Omega^m} \\
&\cdot
 \underline{
(t_1w_3^m t_1^{-1})(a_2a_1^{1-\lambda}a_2^{-1}) f(a_3^2,a_2a_1^2a_2^{-1}) a_3^{1-\lambda}}
\cdot f(a_1^2,t_1w_3t_1^{-1})
\cdot t_1\\
=&
a_1^{1-\lambda} \cdot f(w_3,a_3^2) \cdot
f(a_2a_1^2a_2^{-1},a_1^2) \cdot
f(t_1w_3 t_1^{-1},w_3)\cdot
(a_2a_1^{-2}a_2^{-1}\Omega)^m \cdot
f(a_3^2,a_2a_1^2a_2^{-1})\\
&\cdot f(a_1^2,t_1w_3t_1^{-1})
\cdot a_3^{1-\lambda}
\cdot t_1, \\
\end{align*}
where we applied \eqref{NSlem1.5}
in the above underlined and underwaved portions,
and used
$a_2a_1^{-2}a_2^{-1}\Omega =w_4$.

We then conclude the aimed quantity \eqref{FBt1} as:
\begin{align}
\FB(t_1)=&
a_1^{1-\lambda}f(w_3,a_3^2) 
f(a_2a_1^2a_2^{-1},a_1^2) 
f(t_1w_3t_1^{-1},w_3)
f(a_3^2,a_2a_1^2a_2^{-1})
\\
&\cdot
f(a_1^2,t_1w_3t_1^{-1}) \cdot
a_3^{1-\lambda}
w_4^\lambda \cdot t_1
\nonumber
\\
=&
a_1^{1-\lambda}
a_3^{1-\lambda}
w_4^m
\cdot t_1
\label{afterPentagon}
\end{align}

After \eqref{afterPentagon}, it suffices to derive
$$
a_1^{1-\lambda}
a_3^{1-\lambda}\cdot t_1=w_4^{-m}\cdot t_1^\lambda.
$$
In fact, we have $a_1^2a_3^2t_1^2=w_4$.
Using the fact that $w_4$ commutes with $a_1a_3$
we obtain
$$t_1^\lambda=t_1^{\lambda-1}\cdot t_1
=t_1\cdot(t_1^2)^m
=t_1\cdot w_4^m (a_1a_3)^{1-\lambda}
$$
which completes the proof of our claim.
\end{proof}

Next, we shall show
\begin{prop}
\label{app-prop-t2}
Let $F=(\lambda,f)\in\GT$ and
$\FB$ as in the above.
Then $F_{\mathcal{B}}$ acts on $t_2=a_4a_3a_5a_4$ by
\begin{equation}
\FB(t_2)= f(t_2^2,t_1^2) t_2^\lambda
f(t_1^2,t_2^2).
\end{equation}
\end{prop}

\begin{proof}
Now by definition $\FB$ acts on $a_1,a_2,a_3$ by 
\begin{align}
\FB(a_i)&=a_i^\lambda \quad (i=1,3,5), \\
\FB(a_2)&=f(w_3,a_3^2) f(a_2^2,a_1^2) a_2^\lambda
f(a_1^2,a_2^2)f(a_3^2,w_3) \\
\FB(a_4)&=f(w_5,a_5^2) f(w_3,a_3^2) f(a_4^2,w_4) 
a_4^\lambda
f(w_4,a_4^2)f(a_3^2,w_3) f(a_5^2,w_5) \\
&=f(w_5,a_5^2) f(w_3,a_3^2) f(a_4^2,w_4) 
a_4^{\lambda-1}
\nonumber
f(a_4w_4a_4^{-1},a_4^2) \\ &{\qquad}\cdot f(a_4a_3^2a_4^{-1},a_4w_3a_4^{-1})f(a_4a_5^2a_4^{-1},a_4w_5a_4^{-1})\cdot a_4. 
\nonumber
\end{align}
We compute $\FB(t_2)=\FB(a_4)\FB(a_3)\FB(a_5)\FB(a_4)$
along with the diagram in Figure \ref{fig:AmovesforTi},
where ``$D$'' indicates a disc with two punctures numbered 
$2,1$ from the left and arrows between pants-decompositions
(of a punctured plane) have labels of $f(G,L)$ with
$G,L$ indicate twists of a gained loop, a lost loop
respectively.

\begin{align}
&\FB(t_2)=\FB(a_4)\FB(a_3)\FB(a_5)\FB(a_4) 
\label{FBt2}
\\
=&
f(w_5,a_5^2) 
{f(w_3,a_3^2) f(a_4^2,w_4)} 
a_4^{\lambda-1}
{
f(a_4w_4a_4^{-1},a_4^2)f(a_4a_3^2a_4^{-1},w_3)
} 
\nonumber\\ &{\qquad}\cdot f(a_4a_5^2a_4^{-1},w_5)\cdot 
a_4(a_3a_5)^{\lambda-1}a_4^{-1}\cdot a_4a_3a_5\cdot \FB(a_4) \nonumber\\
=&
f(w_5,a_5^2) \underline{f(w_3,a_3^2) f(a_4^2,w_4)} 
a_4^{\lambda-1}
\uwave{f(a_4w_4a_4^{-1},a_4^2)f(a_4a_3^2a_4^{-1},w_3)} \nonumber\\ 
&{\qquad}\cdot 
{f(a_4a_5^2a_4^{-1},w_5)}
\cdot 
a_4a_3^{\lambda-1}a_4^{-1}\cdot
{a_4a_5^{\lambda-1}a_4^{-1}}
\nonumber \\
&\qquad\cdot
{f(t_2w_5t_2^{-1},a_4a_5^2a_4^{-1})} f(t_2w_3t_2^{-1},a_4a_3^2a_4^{-1}) f(t_2a_4^2t_2^{-1},a_4w_4a_4^{-1}) 
t_2a_4^{\lambda-1}t_2^{-1} \nonumber \\
&{\qquad}\cdot
f(t_2w_4t_2^{-1},t_2a_4^2t_2^{-1})
f(t_2a_3^2t_2^{-1},t_2w_3t_2^{-1})  f(t_2a_5^2t_2^{-1},t_2w_5t_2^{-1})\cdot t_2. 
\nonumber
\end{align}
Replacing the above underlined and underwaved portions 
by:
\begin{align}
\label{twoPentagon_t2}
&f(w_3,a_3^2) f(a_4^2,w_4)
= f((a_3a_4)^3,w_4)f(a_4^2,a_3^2)f(w_3,(a_3a_4)^3),
\\
&f(a_4w_4a_4^{-1},a_4^2)f(a_4a_3^2a_4^{-1},a_4w_3a_4^{-1})
=f((a_3a_4)^3,w_3) f(a_4a_3^2a_4^{-1},a_4^2)
f(a_4w_4a_4^{-1},(a_3a_4)^3),
\nonumber
\end{align}
with cancellation of $f(w_3,(a_3a_4)^3)^{\pm 1}$ there across
$a_4^{\lambda-1}$, and noting the commutativity of
${f(a_4a_5^2a_4^{-1},w_5)}$ and $f(t_2w_4t_2^{-1},t_2a_4^2t_2^{-1})$,
the computation \eqref{FBt2} continues
to:
\begin{align}
=&f(w_5,a_5^2)
f((a_3a_4)^3,w_4)
\uuline{f(a_4^2,a_3^2)a_4^{\lambda-1}
 f(a_4a_3^2a_4^{-1},a_4^2)}
f(a_4w_4a_4^{-1},(a_3a_4)^3) 
\\
&\cdot
\uuline{a_4a_3^{\lambda-1}a_4^{-1}}
f(t_2w_3t_2^{-1},a_4a_3^2a_4^{-1})
f(a_4a_5^2a_4^{-1},w_5)a_4a_5^{\lambda-1}a_4^{-1} 
\nonumber \\
&\cdot \underline{f(t_2w_5t_2^{-1},a_4a_5^2a_4^{-1})
f(t_2a_4^2t_2^{-1}, a_4w_4a_4^{-1})}
t_2a_4^{\lambda-1}t_2^{-1}
\uwave{f(t_2w_4t_2^{-1},t_2a_4^2t_2^{-1})
f(a_3^2,t_2w_5t_2^{-1})} \nonumber \\
&\cdot
f(a_5^2,t_2w_3t_2^{-1})\cdot t_2
\nonumber
\end{align}
Moreover, applying the hexagon relation
\eqref{NSlem1.5}
to the double-lined factors,  and replacing 
the underlined and underwaved portions 
by 5-cyclic relations 
respectively by 
\begin{align*}
&f(a_5^{-1}(a_3a_4)^3a_5,t_2w_3t_2^{-1})
f(t_2a_4^2t_2^{-1},a_4a_5^2a_4^{-1})
f(t_2w_5t_2^{-1},a_5^{-1}(a_3a_4)^3a_5), \\
&f(a_5^{-1}(a_3a_4)^3a_5,t_2w_3t_2^{-1})
f(a_3^2,t_2a_4^2t_2^{-1})
f(t_2w_4t_2^{-1},a_5^{-1}(a_3a_4)^3a_5)
\end{align*}
(with obvious cancellation) in the above, 
we continue
\eqref{FBt2} more to:

\begin{align*}
=& f(w_5,a_5^2)f((a_3a_4)^3,w_4)
a_3^{1-\lambda}
(a_3a_4)^{3m}
\underline{
f(a_4a_3^2a_4^{-1},a_3^2)
f(a_4w_4a_4^{-1},(a_3a_4)^3)
f(t_2w_3t_2^{-1},a_4a_3^2a_4^{-1})}
\\
&\cdot
f(a_4a_5^2a_4,w_5)(a_4a_5^{\lambda-1}a_4^{-1})
f(a_5^{-1}(a_3a_4)^3a_5,t_2w_3t_2^{-1})\\
&\cdot
\uwave{
f(t_2a_4^2t_2^{-1},a_4a_5^2a_4^{-1})
(t_2a_4^{\lambda-1}t_2^{-1})
f(a_3^2,t_2a_4^2t_2^{-1})}\cdot
\uuline{
f(t_2w_4t_2^{-1},a_5^{-1}(a_3a_4)^3a_5)
f(a_5^2,t_2w_3t_2^{-1})
}\cdot t_2
\end{align*}
Use relation (III) to the above underlined and double-underlined portions to get
\begin{align*}
&f(t_2w_3t_2^{-1},(a_3a_4)^3)f(a_4w_4a_4^{-1},a_3^2), \\
&f((a_3a_4a_5)^4,t_2w_3t_2^{-1})
f(a_5^2,a_5^{-1}(a_3a_4)^3a_5)
f(t_2w_4t_2^{-1},(a_3a_4a_5)^4)
\end{align*}
respectively, and apply relation (II) to the underwaved 
portion. Then, noting also that $a_4a_5^2a_4^{-1}$
commutes with $a_5^{-1}(a_3a_4)^3a_5$, $t_2w_3t_2^{-1}$,
we may continue to:
\allowdisplaybreaks
\begin{align*}
\FB(t_2)=& f(w_5,a_5^2)f((a_3a_4)^3,w_4)
a_3^{1-\lambda}
(a_3a_4)^{3m}
f(t_2w_3t_2^{-1},(a_3a_4)^3)
\underline{f(a_4w_4a_4^{-1},a_3^2)}
\\
&\cdot
\underline{
f(a_4a_5^2a_4^{-1},w_5)f(a_5^{-1}(a_3a_4)^3a_5, t_2w_3t_2^{-1})}
(a_5^{-1}a_4^{\lambda-1}a_5) \\
&\cdot
f(t_2a_4^2t_2^{-1},a_4a_5^2a_4^{-1})
(t_2a_4^{\lambda-1}t_2^{-1})
f(a_3^2,t_2a_4^2t_2^{-1})\\
&\cdot
f((a_3a_4a_5)^4,t_2w_3t_2^{-1})
f(a_5^2,a_5^{-1}(a_3a_4)^3a_5)
f(t_2w_4t_2^{-1},(a_3a_4a_5)^4)
\cdot t_2 \\
=&
f(w_5,a_5^2)f((a_3a_4)^3,w_4)
a_3^{1-\lambda}
(a_3a_4)^{3m}
f(t_2w_3t_2^{-1},(a_3a_4)^3) \\
&\cdot
f(a_5^{-1}(a_3a_4)^3a_5,w_5)
\uwave{f(a_5^{-1}a_4^2a_5,a_3^2)
(a_5^{-1}a_4^{\lambda-1}a_5)
f(t_2a_4^2t_2^{-1},a_4a_5^2a_4^{-1})}
(t_2a_4^{\lambda-1}t_2^{-1})
\\
&\cdot
f(a_3^2,t_2a_4^2t_2^{-1})
f((a_3a_4a_5)^4,t_2w_3t_2^{-1})
f(a_5^2,a_5^{-1}(a_3a_4)^3a_5)
f(t_2w_4t_2^{-1},(a_3a_4a_5)^4)
\cdot t_2 \\
=&
f(w_5,a_5^2)f((a_3a_4)^3,w_4)
a_3^{1-\lambda}
(a_3a_4)^{3m}
\uwave{f(t_2w_3t_2^{-1},(a_3a_4)^3)} \\
&\cdot
\uwave{f(a_5^{-1}(a_3a_4)^3a_5,w_5)}
a_3^{1-\lambda}
\uuline{f(t_2a_4^2t_2^{-1},a_3^2)
(t_2a_4^{1-\lambda}t_2^{-1})}
a_5^{-1}(a_3a_4)^{-3m}a_5
\uuline{(t_2a_4^{\lambda-1}t_2^{-1})}
\\
&\cdot
\uuline{f(a_3^2,t_2a_4^2t_2^{-1})}
\uwave{f((a_3a_4a_5)^4,t_2w_3t_2^{-1})}
f(a_5^2,a_5^{-1}(a_3a_4)^3a_5)
f(t_2w_4t_2^{-1},(a_3a_4a_5)^4)
\cdot t_2 
\end{align*}

Cancelling the double-lined factors in the above line
allows us to apply relation (III) to the underwaved part.
Then, we obtain
\allowdisplaybreaks
\begin{align*}
\FB(t_2)=& f(w_5,a_5^2)f((a_3a_4)^3,w_4)
a_3^{1-\lambda}
(a_3a_4)^{3m}\cdot
f((a_3a_4a_5)^4,w_5)\cdot
a_3^{1-\lambda} \\
&\cdot
\underline{f(a_5^{-1}(a_3a_4)^3a_5,(a_3a_4)^3)
(a_5^{-1}(a_3a_4)^{3m}a_5)
f(a_5^2,a_5^{-1}(a_3a_4)^3a_5)}
f(t_2w_4t_2^{-1},(a_3a_4a_5)^4)
\cdot t_2 \\
=& f(w_5,a_5^2)f((a_3a_4)^3,w_4)
f((a_3a_4a_5)^4,w_5) a_3^{2-2\lambda}
(a_3a_4)^{3m}\cdot\\
&\cdot
(a_3a_4)^{-3m}
\cdot
f(a_5^2,(a_3a_4)^3)
a_5^{1-\lambda} \Omega^m
f(t_2w_4t_2^{-1},(a_3a_4a_5)^4)
\cdot t_2 \\
\end{align*}
where $\Omega=a_5^{-1}(a_3a_4)^3a_5\cdot (a_3a_4)^3\cdot
a_3^2=a_3^2\cdot(a_3a_4a_5)^4=a_3^2\cdot (a_3^2a_5^2t_2^2)$.
It follows then that
\begin{align*}
\FB(t_2)=&
a_3^{2-2\lambda}
\uwave{
f(w_5,a_5^2)f((a_3a_4)^3,w_4)
f((a_3a_4a_5)^4,w_5) 
f(a_5^2,(a_3a_4)^3)}
\\
&\cdot
a_5^{1-\lambda}
a_3^{\lambda-1}
(a_3a_4a_5)^{4m}
f(t_2w_4t_2^{-1},(a_3a_4a_5)^4)
\cdot t_2
\\
=&
a_3^{2-2\lambda}
f((a_3a_4a_5)^4,w_4) a_5^{1-\lambda}
a_3^{\lambda-1} (a_3^{\lambda-1}a_5^{\lambda-1}
t_2^{\lambda-1})
f(t_2w_4t_2^{-1},(a_3a_4a_5)^4)
\cdot t_2 \\
=&
f((a_3a_4a_5)^4,w_4)\cdot
t_2^{\lambda-1}\cdot
f(t_2w_4t_2^{-1},(a_3a_4a_5)^4)
\cdot t_2 \\
=&f(t_2^2,t_1^2)\cdot t_2^{\lambda-1}\cdot 
f(t_2t_1^2t_2^{-1},t_2^2)
\cdot t_2
=f(t_2^2,t_1^2)\cdot t_2^{\lambda}\cdot 
f(t_1^2,t_2^2),
\end{align*}
where the equality to the last line follows from
\begin{equation}
f((a_3a_4a_5)^4,w_4)=f(t_2^2,t_1^2),\quad
f(t_2w_4t_2^{-1},(a_3a_4a_5)^4)
=f(t_2t_1^2t_2^{-1},t_2^2)
\end{equation}
which are easily derived from the known identities
$(a_3a_4a_5)^4=a_3^2a_5^2t_2^2$ and $w_4=a_1^2a_3^2t_1^2$.
We thus conclude the desired assertion for
$\FB(t_2)$.
\end{proof}

\begin{prop}
Let $i\ge 2$ and write $W_i:=(t_1\cdots t_{i-1})^i$.
Then, we have
$$
F_{\mathcal{B}}(t_i)=f(t_i^2,W_i)\,t_i^\lambda \,f(W_i,t_i^2).
$$
\end{prop}

\begin{proof}
We compute $\FB(t_i)=\FB(a_{2i})\FB(a_{2i-1})\FB(a_{2i+1})\FB(a_{2i})$
along with the diagram in Figure \ref{fig:AmovesforTi},
where numbering of punctures outside $D$ 
is shifted $+(2i-4)$ so that
$D$ is filled by a disc 
with punctures numbered 
from left $2i-2$ to right $1$ and with
given a pants-decomposition 
corresponding to bracketing $(...((2i-2,2i-3)\cdots((43)(21))...)$.
The computations go the same way as in the case of 
$\FB(t_2)$ (Proposition \ref{app-prop-t2}),
thanks to the commutativity of Dehn twists
along disjoint loops.
\end{proof}

 \end{document}